\documentclass[reqno]{amsart}
\numberwithin{equation}{section}

\usepackage{amsmath}
\usepackage{amsfonts}
\usepackage{amsthm}
\usepackage{amssymb}
\usepackage{graphicx}
\usepackage{graphics}
\usepackage{bm}
\usepackage{dsfont}
\usepackage{color}
\usepackage[font=footnotesize]{caption}
\usepackage[all]{xy}
\usepackage{tikz-cd}
\usepackage{mathtools}
\usepackage{scrextend}
\usepackage{todonotes}
\usepackage{url}
\usepackage{hyperref}\hypersetup{colorlinks}
\usepackage{enumerate}

\definecolor{darkred}{rgb}{1,0,0}
\definecolor{darkgreen}{rgb}{0,0.8,0}
\definecolor{darkblue}{rgb}{0,0,1}
\hypersetup{colorlinks,
linkcolor=darkblue,
filecolor=darkgreen,
urlcolor=darkred,
citecolor=darkgreen}

\newtheorem{thm}{Theorem}[section]

\newtheorem{lem}[thm]{Lemma}
\newtheorem{prop}[thm]{Proposition}

\newtheorem{maintheorem}{Theorem}
\theoremstyle{definition}
\newtheorem{rem}[thm]{Remark}
\newtheorem{exm}[thm]{Example}
\newtheorem{ass}[thm]{Assumption}

\DeclareMathOperator{\area}{\mathrm{area}}
\DeclareMathOperator{\lcm}{\mathrm{lcm}}
\DeclareMathOperator{\const}{\mathrm{const}}

\DeclareMathOperator{\Cal}{Cal}
\DeclareMathOperator{\Flux}{\mathrm{Flux}}
\newcommand{\Ham}{\mathrm{Ham}}

\newcommand{\std}{\mathrm{std}}
\newcommand{\vol}{\mathrm{vol}}

\newcommand{\fix}{\mathrm{fix}}

\newcommand{\pr}{\mathrm{pr}}

\newcommand{\N}{\mathds{N}}

\newcommand{\Z}{\mathds{Z}}
\newcommand{\R}{\mathds{R}}

\newcommand{\C}{\mathds{C}}

\newcommand{\UU}{\mathcal{U}}
\newcommand{\VV}{\mathcal{V}}

\newcommand{\eul}{e}

\newcommand{\diff}{d}

\newcommand{\id}{\mathrm{id}}

\newcommand{\Tan}{\mathrm{T}}

\begin{document}

\title[Higher systolic inequalities]{Higher systolic inequalities\\ for 3-dimensional contact manifolds}

\author[A. Abbondandolo]{Alberto Abbondandolo}
\address{Alberto Abbondandolo\newline\indent Ruhr Universit\"at Bochum, Fakult\"at f\"ur Mathematik\newline\indent Geb\"aude IB 3/65, D-44801 Bochum, Germany}
\email{alberto.abbondandolo@rub.de}

\author[C. Lange]{Christian Lange}
\address{Christian Lange\newline\indent Ludwig-Maximilians-Universit\"at M\"unchen, Mathematisches Institut\newline\indent Theresienstra{\ss}e 39, D-80333 Munich, Germany}
\email{lange@math.lmu.de}

\author[M. Mazzucchelli]{Marco Mazzucchelli}
\address{Marco Mazzucchelli\newline\indent CNRS, UMPA, \'Ecole Normale Sup\'erieure de Lyon\newline\indent 46 all\'ee d'Italie, 69364 Lyon, France}
\email{marco.mazzucchelli@ens-lyon.fr}

\date{July 25, 2021}



\begin{abstract}
A contact form is called Besse when the associated Reeb flow is periodic. We prove that Besse contact forms on closed connected 3-manifolds are the local maximizers of suitable higher systolic ratios. Our result extends earlier ones for Zoll contact forms, that is, contact forms whose Reeb flow defines a free circle action. 
\end{abstract}

\maketitle

\section{Introduction}
\label{s:introduction}

\subsection{Background and main result}
The aim of this paper is to prove some sharp inequalities involving the periods of closed orbits of Reeb flows on 3-manifolds and the contact volume.
Let $Y$ be a closed, connected, orientable 3-manifold. We recall that a one-form $\lambda$ on $Y$ is called a contact form when $\lambda\wedge d\lambda$ is nowhere vanishing. The contact form $\lambda$ induces a vector field $R_{\lambda}$, which is called Reeb vector field of $\lambda$, by the identities $R_{\lambda} \,\lrcorner\, \diff\lambda = 0$ and $R_{\lambda} \,\lrcorner\, \lambda=1$.
The flow of $R_{\lambda}$ is called the Reeb flow, and we will denote it by $\phi_{\lambda}^t$. It preserves the contact form $\lambda$, and in particular the volume form $\lambda\wedge d\lambda$. Reeb flows are also called contact flows in the literature. Reeb flows on 3-manifolds constitute a special class of volume preserving flows with the remarkable feature of always having closed orbits: the Weinstein conjectures postulates that Reeb flows on arbitrary closed contact manifolds admit closed orbits, and this conjecture has been confirmed in dimension 3 by Taubes, see \cite{Taubes:2007wi}.

We denote by $\tau_1(\lambda)$ the minimum of all periods of closed Reeb orbits and define the systolic ratio of $\lambda$ as the quotient
\begin{align}
\label{e:1st_systolic_ratio} 
\rho_1(\lambda) := \frac{\tau_1(\lambda)^2}{\vol(Y,\lambda)},
\end{align}
where the contact volume $\vol(Y,\lambda)$ is defined as the integral of the volume form $\lambda\wedge \diff\lambda$ over $Y$. The choice of the power 2 in the numerator of~\eqref{e:1st_systolic_ratio}  makes $\rho_1$ invariant under rescaling: $\rho_1(c\lambda) = \rho_1(\lambda)$ for every non-zero constant $c$. As observed in \cite[Lemma 2.1]{Croke:1994}, different contact forms on $Y$ inducing the same Reeb vector field  give the same contact volume. Therefore, the systolic ratio $\rho_1$ is a dynamical invariant of Reeb flows. It is actually invariant by smooth conjugacies and linear time rescalings.

The term ``systolic ratio'' is borrowed from metric geometry: the systolic ratio of a Riemannian metric on a closed surface is the ratio between the square of the length of the shortest closed geodesic and the Riemannian area. Geodesic flows are particular Reeb flows, and the metric systolic ratio coincides with $2\pi$-times the contact systolic ratio defined above. Indeed, the length of any closed geodesic agrees with its period as closed Reeb orbit, and the contact volume of the unit tangent bundle of a Riemannian surface is $2\pi$-times the Riemannian area. 

Still borrowing the terminology from Riemannian geometry, a contact form $\lambda$ on $Y$ is called Zoll if all its Reeb orbits are closed and have the same minimal period. In this case, the Reeb flow of $\lambda$ induces a free $S^1$-action on $Y$, and the systolic ratio of $\lambda$ has the value $-1/\eul$, where the negative integer $\eul$ is the Euler number of the $S^1$-bundle which is induced by this $S^1$-action. 

Zoll contact forms are precisely the local maximizers of the systolic ratio $\rho_1$ in the $C^3$-topology of contact forms: this was proven for arbitrarily closed 3-manifolds by Benedetti and Kang in \cite{Benedetti:2021aa}, generalizing a result of the first author together with Bramham, Hryniewicz and Salom\~ao in \cite{Abbondandolo:2018fb} for the 3-sphere. Recently, this result has been extended to manifolds of arbitrary dimension by the first author and Benedetti, see \cite{Abbondandolo:2019tl}. We refer the reader to the latter paper and to \cite{Alvarez-Paiva:2014uq} for a discussion on some consequences of the local systolic maximality of Zoll contact forms in metric and systolic geometry.

We denote by $\sigma(\lambda)$ the action spectrum (or period spectrum) of the Reeb flow of $\lambda$, i.e.\ the set
\[
\sigma(\lambda)=\big\{t>0\ \big|\ \fix(\phi_\lambda^t)\neq\varnothing\big\}.
\]
Note that every closed Reeb orbit contributes to $\sigma(\lambda)$ with all the multiples of its minimal period. In general, $\sigma(\lambda)$ is a non-empty closed set of Lebesgue measure zero, and for generic contact forms it is discrete. 

The number $\tau_1(\lambda)$ is the minimum of $\sigma(\lambda)$, and we would like to define $\tau_k(\lambda)$ as the $k$-th element of $\sigma(\lambda)$, where the elements of $\sigma(\lambda)$ are ordered increasingly and are counted with multiplicity given by the number of closed orbits having a given period. Since in general $\sigma(\lambda)$ is not discrete, a correct definition of $\tau_k(\lambda)$ is the following: $\tau_k(\lambda)$ is the infimum of all positive real numbers $\tau$ such that there exist at least $k$ closed Reeb orbits with period less than or equal to $\tau$; here, each iterate of a closed Reeb orbit contributes to the count. In formulas,
\begin{equation}
\label{e:tau_k}
\tau_k(\lambda):=\inf\Bigg\{\tau>0\ \Bigg|\ \sum_{0<t\leq\tau} \#\big(\fix(\phi_\lambda^t)/\sim\big)\geq k  \Bigg\},
\end{equation}
where $\sim$ is the equivalence relation on $Y$ identifying points on the same Reeb orbit, i.e.\ $z_0\sim z_1$ if and only if $z_1=\phi_\lambda^t(z_0)$ for some $t\in\R$. Note that the sequence of values $\tau_k(\lambda)$, $k\geq1$, is (not necessarily strictly) increasing and consists of elements of $\sigma(\lambda)$.  

If $\sigma(\lambda)$ is discrete and for any $\tau \in \sigma(\lambda)$ there are finitely many Reeb orbits of period $\tau$, then $k\mapsto \tau_k(\lambda)$ is a surjective map from $\N$ to $\sigma(\lambda)$. If instead there are infinitely many periodic orbits of (not necessarily minimal) period $\tau_{k}(\lambda)$ for some $k$, or a strictly decreasing sequence in $\sigma(\lambda)$ converging to $\tau_{k}(\lambda)$, then $\tau_h(\lambda)=\tau_{k}(\lambda)$ for every $h\geq k$. 

We now define the $k$-th systolic ratio of the contact form $\lambda$ as the positive number
\[
\rho_k(\lambda) := \frac{\tau_k(\lambda)^2}{\vol(Y,\lambda)}.
\]
The aim of this paper is to give a complete characterization of local maximizers of the $k$-th systolic ratio $\rho_k$.

Borrowing once more the terminology from Riemannian geometry, a contact form $\lambda$ on $Y$ is called Besse if all its Reeb orbits are closed. Here, different Reeb orbits are not required to have the same minimal period, and therefore Besse contact forms constitute a larger class than Zoll forms. Thanks to a theorem of Wadsley \cite{Wadsley:1975sp} or, in the special case of dimension 3, an earlier theorem of Epstein \cite{Epstein:1972}, Besse Reeb flows are periodic (see also \cite{Sullivan:1978}). In our case, since $Y$ has dimension 3, Epstein's theorem implies that all Reeb orbits of the Besse contact form $\lambda$ have the same minimal period $T$ except for finitely many ones, whose minimal period divides $T$. The orbits of the first kind are called regular, whereas the finitely many exceptional orbits with smaller minimal period are called singular.

In Riemannian geometry, suitable lens spaces have a geodesic flow that is Besse but not Zoll. Nevertheless, on simply connected manifolds, Besse geodesic flows are conjectured to be Zoll: this was confirmed for the 2-sphere, thanks to a classical result of Gromoll and Grove \cite{Gromoll:1981kl}, and for $n$-spheres of dimension $n\geq 4$, by a recent result of Radeschi and Wilking \cite{Radeschi:2017dz}. In the more general class of Finsler geodesic flows, and in the even larger class of Reeb flows, there are plenty of examples of flows that are Besse but not Zoll: the simplest ones are the geodesic flows of rational Katok's Finsler metrics on the 2-sphere, see \cite{Katok:1973mw,Ziller:1983rw}, and the Reeb flows on rational ellipsoids in $\C^2$; other examples are the geodesic flows on certain Riemannian orbifolds, see \cite{Besse:1978pr,Lange:2020,Lange:2021}.

The theory of Seifert fibrations leads to the construction of many more examples and to a full classification of Besse Reeb flows in dimension 3, see \cite{Kegel:2021vo} and Section \ref{s:surgery} below. Indeed, the Reeb flow of a Besse contact form $\lambda$ on $Y$ induces a locally free $S^1$-action, whose quotient projection $\pi: Y \rightarrow B$ is a Seifert fibration over a 2-dimensional orbifold $B$. The Euler number $\eul$ of such a Seifert fibration  is rational and negative, see \cite{Lisca:2004oz}, and conversely any Seifert fibration with negative Euler number can be realized in this way. Moreover
\begin{equation}
\label{volumeBesse}
\vol(Y,\lambda) = -T^2 \eul,
\end{equation}
where $T$ is the minimal common period of the Reeb orbits of $\lambda$, see \cite[Cor.~6.3]{Geiges:2020aa} or Lemma \ref{l:volume_Besse} below.

If $\lambda$ is a Besse contact form on the closed 3-manifold $Y$,  the sequence $\tau_k(\lambda)$ which we introduced above stabilizes: denoting by $T$ the minimal common period of the Reeb orbits, by $\gamma_1,\dots,\gamma_h$ the singular Reeb orbits, and by $\alpha_1,\dots,\alpha_h$ the integers greater than 1 such that $\gamma_i$ has minimal period $T/\alpha_i$, we find that $\tau_k(\lambda)=T$ for every $k\geq k_0(\lambda)$, where 
\[
k_0(\lambda):= \alpha_1+\dots+\alpha_h - h + 1,
\]
and $k_0(\lambda)$ is the minimal integer with this property. Indeed, the Reeb flow of $\lambda$ has a continuum of orbits of minimal period $T$ and precisely $\alpha_1+\dots+\alpha_h - h$ orbits of period strictly less than $T$, given by the iterates $\gamma_i^j$ for $1\leq j \leq \alpha_i - 1$ of the singular orbits. 

Together with (\ref{volumeBesse}), the above considerations yield the following formula for the $k_0(\lambda)$-th systolic ratio of the Besse contact form $\lambda$:
\[
\rho_{k_0(\lambda)} (\lambda) = - \frac{1}{\eul},
\]
where $\eul$ is the Euler number of the Seifert fibration $\pi:Y\to B$ induced by $\lambda$.

We now state the main result of this paper, which characterizes Besse contact forms as local maximizers of the higher systolic ratios.

\begin{maintheorem}
\label{t:main}
Let $Y$ be a closed, connected, orientable 3-manifold and $k$ a positive integer. 
\begin{itemize}
\item[$(i)$] If a contact form $\lambda_0$ on $Y$ is a local maximizer of the $k$-th systolic ratio $\rho_k$ in the $C^{\infty}$-topology, then $\lambda_0$ is Besse with $k_0(\lambda_0)=k$.

\item[$(ii)$] Every Besse contact form $\lambda_0$ on $Y$ such that $k_0(\lambda_0)=k$ has a $C^3$-neigh\-bor\-hood $\UU$ in the space of contact forms on $Y$ such that
\[
\rho_k(\lambda) \leq \rho_k(\lambda_0), \qquad \forall \lambda\in \UU,
\]
with equality if and only if there exists a diffeomorphism $\theta:Y\to Y$ such that $\theta^*\lambda=c\lambda_0$ for some $c>0$.

\end{itemize}
\end{maintheorem}

We remark that Besse contact forms are never global maximizers of $\rho_k$ on the space of contact forms inducing a given contact structure $\xi$ on the closed 3-manifold $Y$: indeed, $\rho_k\geq \rho_1$ and $\rho_1$ is unbounded from above on the space of all contact forms on $(Y,\xi)$. See \cite{Abbondandolo:2019ta} for the case of 3-dimensional contact manifolds and \cite{Saglam:2018tb} for the general case.

\begin{exm}
It is instructive to consider Theorem~\ref{t:main} in the case $Y=S^3$. Any Besse contact form on $S^3$ coincides, up to a diffeomorphism and multiplication by a positive number, with the restriction of the standard Liouville 1-form
\[
\lambda_0:= \frac{1}{2} \sum_{j=1}^2 \big(x_j \, dy_j - y_j \, dx_j \big)
\]
of $\R^4$ to the boundary of the solid ellipsoid
\[
E(p,q):= \biggl\{ z\in \C^2 \ \bigg|\ \frac{ |z_1|^2}{p} + \frac{ |z_2|^2}{q} \leq \frac1\pi \biggr\} \subset \C^2 = \R^4,
\]
where $p\leq q$ are coprime positive integers, see for instance \cite[Prop.~5.2]{Geiges:2018aa} and \cite[Th.~1.1]{Mazzucchelli:2020aa}. The Reeb flow of the contact form 
\[
\lambda_{p,q}:= \lambda_0|_{\partial E(p,q)}
\]
has a closed orbit of minimal period $p$, a closed orbit of minimal period $q$ and all other orbits have minimal period $pq$. Therefore,
\[
k_0(\lambda_{p,q}) = p + q - 1,
\]
and, for  $k_0:= k_0(\lambda_{p,q})$,
\[
\rho_{k_0} (\lambda_{p,q}) = pq .
\] 
In particular, $k_0(\lambda_{1,k}) = k$ and, according to Theorem~\ref{t:main}, for every $k\geq 1$ the contact form $\lambda_{1,k}$ is a local maximizer of $\rho_k$. For $k=1,2,3,5$, this is the only local maximizer of $\rho_k$ on $S^3$, but for all the other values of the positive integer $k$ the linear Diophantine equation $ p + q - 1 = k$ is easily seen to have more positive solutions $p\leq q$ that are coprime. For instance, $\rho_4$ is locally maximized by both $\lambda_{1,4}$ and $\lambda_{2,3}$, with $\rho_4(\lambda_{1,4})=4$ and $\rho_4(\lambda_{2,3})=6$. The number of local mazimizers of $\rho_k$ on contact forms on $S^3$ diverges for $k\rightarrow \infty$.
\hfill\qed
\end{exm}

\begin{exm}\label{exm:besse_spindle}
Other natural applications of Theorem \ref{t:main} concern geodesic flows on Riemannian 2-orbifolds. Consider for instance the spindle orbifold $S^2(m,n)$ whose underlying space is $S^2$ and which has two conic singularities of order $m$ and $n$, respectively. Here, $m$ and $n$ are positive integers and a conic singularity of order $m$ corresponds to the local model $\R^2/\Z_m$, where the cyclic group $\Z_m$ acts by rotations. The case $m=n=1$ gives us the standard smooth 2-sphere. Let us assume $m+n>2$, so that we have at least one singular point. The geodesic flow of any Riemannian metric on $S^2(m,n)$ can be seen as a smooth Reeb flow on the lens space $L(m+n,1)$, i.e.\ the quotient of $S^3\subset \C^2$ by the free action of $\Z_{m+n}$ which is generated by the diffeomorphism
\[
(z_1,z_2) \mapsto \bigl( e^{\frac{2\pi i}{m+n}} z_1,  e^{\frac{2\pi i}{m+n}} z_2 \bigr),
\]
see \cite{Lange:2020}.  The spindle orbifold $S^2(m,n)$ admits a Besse Riemannian metric turning it into a Tannery surface: the spindle orbifold is realized as a sphere of revolution having the two cone singularities at the poles, see \cite[Chapter 4]{Besse:1978pr}. The equator is a closed geodesic of length $2\pi$ and all other geodesics are closed with length $2\pi a$, where $a:= m+n$ if $m+n$ is odd and $a:= \frac{m+n}{2}$ if $m+n$ is even. Here, meridians are seen as geodesic segments belonging to closed geodesics of length $2\pi a$. 

The geodesic flow of this Tannery surface has two periodic orbits of minimal period $2\pi$, corresponding to the two orientations of the equator, and all other orbits are closed with minimal period $2\pi a$. Therefore, the integer $k_0$  associated with the corresponding Besse contact form on $L(m+n,1)$ is 
\[
k_0 := 2a -1.
\]
The Tannery surface is a local maximizer in the $C^3$-topology of Riemannian metrics on $S^2(m,n)$ of the $k_0$-th systolic ratio given by the square of the length of the $k_0$-th shortest closed geodesic, where closed geodesics are counted with multiplicity as in (\ref{e:tau_k}), and the Riemannian area of the orbifold. In other words, if the Riemannian metric of the Tannery surface is modified by a $C^3$-small perturbation not affecting the Riemannian area, then the new geodesic flow is either still Besse, and in this case is smoothly conjugate to the Tannery geodesic flow, or the following holds: if the closed geodesic which is obtained by continuation from the equator (which is non-degenerate in the case $m+n>2$ we are considering here) is not shorter than $2\pi$, then there exists a closed geodesic of minimal length close to $2\pi a$ and smaller than this number.   

An analogous result holds for Finsler perturbations of the Tannery surface, where now the two closed geodesics which are obtained by continuation from the equator might be geometrically distinct and have different lengths, if the Finsler perturbation is not reversible. 

Actually, the second author and Soethe \cite{Lange:2021} proved that, within the class of Riemannian rotationally symmetric spindle 2-orbifolds, the Besse ones are even the global maximizers of the suitable higher systolic ratio. 
\hfill\qed
\end{exm}

\subsection{Sketch of the proof of Theorem \ref{t:main}}
We conclude this introduction by giving an informal sketch of the proof of Theorem \ref{t:main}. 

The proof of statement (i) is elementary. First we show that all the Reeb orbits of a contact form $\lambda_0$ which locally maximize $\rho_k$ are closed and have minimal period not exceeding $\tau_k(\lambda_0)$: if there is a point $x\in Y$ whose orbit violates this assertion, we can deform $\lambda_0$ in a neighborhood of $x$ and make the volume smaller without introducing closed orbits of period smaller than $\tau_k(\lambda_0)$. This shows that $\lambda_0$ is Besse with $k_0(\lambda_0)\leq k$. It remains to show that a Besse contact form $\lambda$ does not locally maximize $\rho_k$ if $k> k_0(\lambda_0)$. This can be done by considering explicit perturbations of $\lambda_0$ of the form $(1+\epsilon\, h\circ \pi) \lambda_0$, where $\pi: Y \rightarrow B$ is the quotient projection induced by the locally free $S^1$-action given by the Reeb flow of $\lambda_0$ and $h$ is a suitable smooth real function on $B$.  

The proof of statement (ii) is based on global surfaces of section and on a quantitative fixed point theorem for Hamiltonian diffeomorphisms of compact surfaces that are close to the identity. This kind of arguments has already been used in \cite{Abbondandolo:2018fb, Benedetti:2021aa} in order to prove that Zoll contact forms are local maximizers of $\rho_1$ on closed 3-manifolds, but here we need two new ingredients which may be of independent interest.

We sketch the argument in the case of a Besse contact form $\lambda_0$ that is not Zoll, and hence $k_0:= k_0(\lambda_0)>1$, but in the detailed proof we give in Section \ref{s:proof} we shall recover also the case in which $\lambda_0$ is Zoll. In this paper, by a global surface of section for the flow of the Reeb vector field $R_{\lambda}$ we mean 
a smooth map
$\iota: \Sigma \rightarrow Y$
from an oriented compact surface $\Sigma$ whose restriction to each component of the boundary $\partial\Sigma$ is a positive covering of some periodic orbit of $R_{\lambda}$, whose restriction to the interior of $\Sigma$ is an embedding into $Y\setminus\iota(\partial\Sigma)$ transversal to $R_{\lambda}$, and such that every orbit of $R_{\lambda}$ intersects $\iota(\Sigma)$ in positive and negative time. The first new ingredient is the following result.

\begin{maintheorem}
\label{mt:surface}
If $\lambda_0$ is a Besse contact form on  the closed 3-manifold $Y$ and $\gamma$ is any orbit of $R_{\lambda_0}$, then the Reeb flow of $\lambda_0$ admits a global surface of section $($as in the previous paragraph$\,)$ with $\iota(\partial \Sigma) = \gamma$. 
\end{maintheorem}

See Theorem~\ref{t:sos_Besse} below for a more detailed statement. We remark that the boundary of $\Sigma$ may have several components, but they are all mapped onto $\gamma$ by $\iota$. See also \cite{Albach:2021wi} for related results about global surfaces of section for general flows on 3-manifolds defining a Seifert fibration.

We normalize $\lambda_0$ so that all its regular orbits have minimal period 1, that is, $\tau_{k_0}(\lambda_0)=1$. We apply Theorem \ref{mt:surface} to some singular orbit $\gamma_1$ of period $1/\alpha_1$ of the Reeb flow of $\lambda_0$, which we fix once and for all. The embedded surface $\iota(\mathrm{int}(\Sigma))$ intersects each regular orbit of $R_{\lambda_0}$ exactly $\alpha$ times, for some $\alpha\in \N$ which can be derived from the invariants of the Seifert fibration induced by $\lambda_0$. 

Now consider a contact form $\lambda$ which is suitably close to $\lambda_0$. Since the singular orbits of Besse Reeb flows are non-degenerate, the Reeb flow of $\lambda$ has a closed orbit which is close to $\gamma_1$. Up to multiplying $\lambda$ by a constant and applying a diffeomorphism to it, we can assume that $R_{\lambda}$ coincides with $R_{\lambda_0}$ on $\gamma_1$, which is therefore a closed orbit of both flows, with the same period $1/\alpha_1$. In this case, we can show that $\iota:\Sigma \rightarrow Y$ is a global surface of section also for the Reeb flow of $\lambda$, provided that $\lambda$ is close enough to $\lambda_0$.

We now consider the diffeomorphism
\[
\phi: \Sigma\rightarrow \Sigma
\]
which is given by the $\alpha$-th iterate of the first return map of the flow of $R_{\lambda}$ to $\Sigma$. This map is actually defined only in the interior of $\Sigma$, but we will show that it extends to a diffeomorphism on $\Sigma$. The exact smooth 2-form $\omega:= \iota^*(d\lambda)$ is symplectic in the interior of $\Sigma$ and vanishes with order 1 on the boundary. The map $\phi$ is an exact symplectomorphism on $(\Sigma,\omega)$ and actually
\[
\phi^* \lambda - \lambda = d\tau,
\]
where $\tau: \Sigma \rightarrow (0,+\infty)$ is the $\alpha$-th return time of the flow of $R_{\lambda}$ (or, more precisely, the smooth extension to $\Sigma$ of this function, which is defined in the interior of $\Sigma$). The volume of $(Y,\lambda)$ can be recovered by $\tau$ thanks to the identity
\[
\vol(Y,\lambda) = \frac{1}{\alpha} \int_{\Sigma} \tau\, \omega.
\]
The exact symplectomorphism $\phi$ lifts to a unique element $\tilde\phi$ of $\widetilde{\mathrm{Ham}}_0(\Sigma,\omega)$ which is $C^1$-close to the identity. Here, $\widetilde{\mathrm{Ham}}_0(\Sigma,\omega)$ denotes the subgroup of the universal cover of the group of Hamiltonian diffeomorphisms of $(\Sigma,\omega)$ consisting of isotopy classes $[\{\phi_t\}]$ starting at the identity which have vanishing flux on any curve connecting pairs of points on $\partial \Sigma$. The zero flux condition is important here and holds because we are considering a global surface of section with boundary on just one closed orbit.

Elements $\tilde{\psi}$ of $\widetilde{\mathrm{Ham}}_0(\Sigma,\omega)$ have a well-defined action
\[
a_{\tilde{\psi},\nu} : \Sigma \rightarrow \R, \qquad \psi^* \nu - \nu = d a_{\tilde{\psi},\nu},
\]
with respect to any primitive $\nu$ of $\omega$, where $\psi$ denotes the projection of $\tilde{\psi}$ to the Hamiltonian group. The action at contractible fixed points is independent of $\nu$, and so is the integral of the action on $(\Sigma,\omega)$, which defines the normalized Calabi invariant of $\tilde{\psi}$, i.e.\ the number
\[
\widehat{\Cal}(\tilde{\psi}) := \frac{1}{\area(\Sigma,\omega)} \int_{\Sigma} a_{\tilde{\psi},\nu} \, \omega.
\]
In the case of the lift $\tilde{\phi}$ of the $\alpha$-th return map $\phi$ and of the primitive $\nu:= \iota^*\lambda$ of $\omega$, we obtain the identities
\begin{equation}
\label{intro1}
a_{\tilde{\phi},\nu} = \tau - 1, \qquad \widehat{\Cal}(\tilde{\phi}) = \frac{\vol(Y,\lambda)}{\vol(Y,\lambda_0)} - 1.
\end{equation}

The second new ingredient of this paper is the following fixed point theorem.

\begin{maintheorem}
\label{mt:fixed}
Let $\omega$ be a smooth exact 2-form on the compact surface $\Sigma$ which is symplectic in the interior and vanishes with order 1 on the boundary. For every $c>0$ there exists a $C^1$-neighborhood $\mathcal{U}\subset\widetilde{\mathrm{Ham}}_0(\Sigma,\omega)$ of the identity such that every $\tilde{\psi}\in\mathcal{U}$ with $\widehat{\Cal}(\tilde{\psi})\leq 0$ has a contractible interior fixed point $z$ such that
\[
a_{\tilde{\psi}}(z) + c \, a_{\tilde{\psi}}(z)^2 \leq \frac{1}{2}\, \widehat{\Cal}(\tilde{\psi}),
\]
with the equality holding if and only if $\tilde{\psi}$ is the identity.
\end{maintheorem}

 See Theorem \ref{t:fixed_point} below and the discussion preceding it for the precise definition of all the notions involved in this theorem. The novelty here is the presence of the term which is quadratic in the action. Indeed, the weaker inequality without that term is proven in \cite{Abbondandolo:2018fb} when $\Sigma$ is the disk and in \cite{Benedetti:2021aa} when $\Sigma$ has just one boundary component, but the case of more boundary components can be taken care of similarly thanks to the zero-flux assumption. The constant $\frac{1}{2}$ is sharp in the above inequality, and the presence of the quadratic term is crucial in the conclusion of the argument that we sketch below. 

Since we are assuming that $\gamma_1$ is a closed orbit of $R_{\lambda}$ with minimal period $1/\alpha_1$, and since all the other singular orbits of $R_{\lambda_0}$ correspond to closed orbits of $R_{\lambda}$ of nearby period, the strict inequality $\rho_{k_0}(\lambda)<\rho_{k_0}(\lambda_0)$ holds trivially when $\vol(Y,\lambda) > \vol(Y,\lambda_0)$. Therefore, we can assume that $\vol(Y,\lambda) \leq \vol(Y,\lambda_0)$, which by (\ref{intro1}) implies $\widehat{\Cal}(\tilde{\phi})\leq 0$.
If $\lambda$ is $C^3$-close to $\lambda_0$, then $\tilde{\phi}$ is $C^1$-close to the identity and from Theorem \ref{mt:fixed} with $c=\frac{1}{2}$ we obtain the existence of an interior contractible fixed point $z$ of $\tilde\phi$ with
\begin{equation}
\label{intro2}
a_{\tilde{\phi}}(z) + \frac{1}{2} \, a_{\tilde{\phi}}(z)^2 \leq \frac{1}{2} \,\widehat{\Cal}(\tilde{\phi}).
\end{equation}
By (\ref{intro1}), the fixed point $z$ corresponds to a closed orbit $\gamma\neq \gamma_1^{\alpha_1}$ of $R_{\lambda}$ with (not necessarily minimal) period
\[
\tau(z) = 1 +  a_{\tilde{\phi}}(z).
\]
Since $\tau(z)$ is close to 1, this orbit is either the $\beta$-th iterate of the orbit of $R_{\lambda}$ corresponding to some singular orbit of $R_{\lambda_0}$ of minimal period $1/\beta$ other than $\gamma_1$, or is an orbit of minimal period $\tau(z)$ bifurcating from the set of regular orbits of $R_{\lambda_0}$. In both cases, its presence implies that $\tau_{k_0}(\lambda) \leq \tau(z)$ and by (\ref{intro2}) we find
\[
\begin{split}
\rho_{k_0}(\lambda) &= \frac{\tau_{k_0}(\lambda)^2}{\vol(Y,\lambda)} \leq \frac{\tau(z)^2}{\vol(Y,\lambda)} = \frac{(1+a_{\tilde\phi}(z))^2}{\vol(Y,\lambda)} 
=\frac{1+2a_{\tilde\phi}(z)+a_{\tilde\phi}(z)^2}{\vol(Y,\lambda)} \\ &\leq
\frac{1+\frac{\vol(Y,\lambda)}{\vol(Y,\lambda_0)}-1}{\vol(Y,\lambda)}
= \frac{1}{\vol(Y,\lambda_0)} = \frac{\tau_{k_0}(\lambda_0)^2}{\vol(Y,\lambda_0)} = \rho_{k_0}(\lambda_0).
\end{split}
\]
This shows that $\lambda_0$ is a local maximizer of $\rho_{k_0}$ in the $C^3$-topology. Finally, if this inequality is an equality, then the equality holds in (\ref{intro2}) and hence $\tilde{\phi}$ is the identity. This implies that $\lambda$ is Besse with regular orbits having minimal period 1, and from the local rigidity of Seifert fibrations and Moser's trick we obtain a diffeomorphism $\theta: Y\rightarrow Y$ such that $\theta^* \lambda = \lambda_0$.
This concludes the sketch of the proof of Theorem~\ref{t:main}. 

\subsection{Organization of the paper}
In Section \ref{s:fixed_point}, we review the notions of flux, action and Calabi invariant for symplectomorphisms of surfaces and prove Theorem \ref{mt:fixed}. In Section \ref{s:sos}, we prove Theorem \ref{mt:surface} and show how the resulting global surface of section survives to small perturbations of the contact form. In Section \ref{s:proof}, we prove Theorem \ref{t:main}.

\subsection{Acknowledgments}
We thank Hansj\"org Geiges and Umberto Hryniewicz for discussions concerning surfaces of section, and Gabriele Benedetti for discussing with us the fixed point theorem in \cite{Benedetti:2021aa}. 

A.\ Abbondandolo and M.\ Mazzucchelli are partially supported by the IEA-International Emerging Actions project IEA00549 from CNRS. A.\ Abbondandolo is also partially supported by the SFB/TRR 191 `Symplectic Structures in Geometry, Algebra and Dynamics', funded by the DFG (Projektnummer 281071066 -- TRR 191).

\section{A fixed point theorem}
\label{s:fixed_point}

In this section, we prove a refinement of a fixed point theorem due to Benedetti-Kang \cite[Section~4.4]{Benedetti:2021aa}. Our version allows us to deal with compact surfaces with possibly disconnected boundary and gives a more precise upper bound on the action of the fixed point, which will play a crucial role in the proof of Theorem~\ref{t:main}.

\subsection{Preliminaries: action, flux, and Calabi homomorphism}
Before stating the theorem, we review some facts about the action of exact symplectomorphisms, the flux and the Calabi homomorphism in a setting which is slightly different than the one considered in classical references such as \cite{Calabi:1970,Banyaga:1978,Banyaga:1997,McDuff:1998uu}.

Throughout this section, we consider a compact connected surface $\Sigma$ with non-empty boundary and an exact two-form $\omega$ on $\Sigma$ which is symplectic (i.e.\ nowhere vanishing) in the interior of $\Sigma$. In the fixed points theorem below, we will assume that $\omega$ vanishes on the boundary of $\Sigma$ in a certain precise way, but in order to introduce the objects this theorem is about we do not need this assumption. As we shall see in Section~\ref{s:sos_Besse}, allowing symplectic forms to vanish on the boundary is important when dealing with global surfaces of section of Reeb flows, see also \cite{Abbondandolo:2018fb,Benedetti:2021aa} and, for a more general approach in any dimension, the theory of ideal Liouville domains in \cite{Giroux:2020}.

By a symplectomorphism of $(\Sigma,\omega)$ we mean a diffeomorphism $\phi: \Sigma \rightarrow \Sigma$ such that $\phi^* \omega = \omega$. In other words, $\phi$ is a diffeomorphism of $\Sigma$ which restricts to a symplectomorphism of the open symplectic manifold $\mathrm{int}(\Sigma)$. 

Let $\{\phi_t\}_{t\in [0,1]}$ be an isotopy on $\Sigma$ starting at the identity; we will always tacitly require that every $\phi_t:\Sigma\to\Sigma$ is surjective (i.e.\ a diffeomorphism, and not simply an embedding). We denote by $X_t$ the generating vector field, which is uniquely determined by the equation
\[
\frac{d}{dt} \phi_t = X_t\circ\phi_t.
\]
The isotopy $\{\phi_t\}$ consists of symplectomorphisms if and only if the one-form $X_t \lrcorner \,\omega$ is closed for every $t\in [0,1]$. When these one-forms are exact, i.e.\
\[
X_t \lrcorner \,\omega = dH_t, \qquad \forall t\in [0,1],
\]
for some $H\in C^{\infty}([0,1]\times \Sigma)$, then $X_t$ is called a Hamiltonian vector field, $\{\phi_t\}$ a Hamiltonian isotopy, and $H$ a generating Hamiltonian. Generating Hamiltonians are uniquely defined up to the addition of a function of $t$. The fact that $X_t$ is tangent to the boundary of $\Sigma$ forces each $H_t$ to be constant on each boundary component. By adding a suitable function of time, we could assume that $H_t$ vanishes on a chosen component of the boundary of $\Sigma$, but in general $H_t$ will not necessarily vanish on the other components.

Note that a smooth function $H:\Sigma\to\R$ defines a vector field on the interior of $\Sigma$ through the identity $X  \lrcorner \,\omega = dH$, but in general one needs further assumptions on $H$ in order to guarantee that $X$ extends smoothly to the boundary of $\Sigma$. This will not be a reason of concern for us here, as we will construct Hamiltonians from vector fields and not the other way around.

A symplectomorphism $\phi: \Sigma \rightarrow \Sigma$ is said to be Hamiltonian if $\phi=\phi_1$ for some Hamiltonian isotopy $\{\phi_t\}$. If $\{\phi_t\}$ and $\{\psi_t\}$ are Hamiltonian isotopies generated by the vector fields $X_t$ and $Y_t$ with Hamiltonians $H_t$ and $K_t$, then the composition $\{\psi_t \circ \phi_t\}$ is generated by the vector field $Y_t + (\psi_t)_* X_t$, which is Hamiltonian with generating Hamiltonian 
\begin{equation}
\label{hamprod}
K_t + H_t \circ \psi_t^{-1}.
\end{equation}
Therefore, Hamiltonian diffeomorphisms form a group, which we denote by \linebreak $\Ham(\Sigma,\omega)$. Note that we are not requiring the diffeomorphisms in $\Ham(\Sigma,\omega)$ to be supported in the interior of $\Sigma$.

Every Hamiltonian diffeomorphism $\phi$ is exact, meaning that the one-form $\phi^* \nu - \nu$ is exact for one (and hence any) primitive $\nu$ of $\omega$. Indeed, every isotopy $\phi_t: \Sigma \rightarrow \Sigma$ with $\phi_0=\mathrm{id}$ is Hamiltonian if and only if it is exact for every $t$, see \cite[Proposition 9.3.1]{McDuff:1998uu}. A function $a: \Sigma \rightarrow \R$ satisfying
\[
\phi^* \nu - \nu = da
\]
is called action of the Hamiltonian diffeomorphism $\phi$ with respect to the primitive $\nu$ of $\omega$. Once a primitive of $\omega$ has been fixed, the action is uniquely determined up to an additive constant. If $\phi=\phi_1$ where $\{\phi_t\}$ is a Hamiltonian isotopy with generating Hamiltonian $H_t$, then the formula
\begin{equation}
\label{canaction}
a_{H,\nu}(z) := \int_{\{t\mapsto \phi_t(z)\}} \nu + \int_0^1 H_t \bigl(\phi_t(z)\bigr)\, dt, \qquad \forall z\in \Sigma,
\end{equation}
defines an action of $\phi$ with respect to $\nu$.   

If $\{\phi_t\}$ is a symplectic isotopy starting at the identity and generated by the vector field $X_t$ and $\gamma: [0,1] \rightarrow \Sigma$ a smooth curve, the flux of $\{\phi_t\}$ through $\gamma$ is defined as the symplectic area swept out by the path $\gamma$ under the isotopy $\{\phi_t\}$, i.e.\ the quantity
\[
\begin{split}
\Flux(\{\phi_t\})(\gamma) &:= \int_{[0,1]\times [0,1]} h^* \omega = \int_0^1 \int_0^1 \omega\bigl(X_t(\phi_t(\gamma(s))),d\phi_t(\gamma(s)) [ \dot{\gamma}(s) ] \bigr)\, ds \, dt \\ &=   \int_0^1 \int_{[0,1]} \gamma^* \bigl( (\phi_t^* X_t) \,\lrcorner\, \omega \bigr)\, dt, 
\end{split}
\]
where $h(t,s) := \phi_t(\gamma(s))$ and in the last identity we have used the fact that the diffeomorphisms $\phi_t$ are symplectic. The fact that the one-forms $(\phi_t^*X_t)\lrcorner\,\omega$ are closed implies that $\Flux(\{\phi_t\})(\gamma)$ only depends on the homotopy class of $\gamma$ relative to the endpoints, or on the free homotopy class of the closed curve $\gamma$.

Moreover, if $\nu$ is a primitive of $\omega$ we find by Stokes theorem
\[
\Flux(\{\phi_t\})(\gamma) = \int_{\gamma} (\phi_1^* \nu-\nu) + \int_{\{t \mapsto \phi_t(\gamma(0))\}} \nu -   \int_{\{t \mapsto \phi_t(\gamma(1))\}} \nu.
\]
The above identity shows that if $\gamma$ is a curve joining two points on the boundary of $\Sigma$, then $\Flux(\{\phi_t\})(\gamma)$ does not vary under homotopies of $\{\phi_t\}$ fixing the endpoints. If $\gamma$ is a closed curve, then $\Flux(\{\phi_t\})(\gamma)$ depends only on the homology class of the closed one-form $\phi_1^* \nu - \nu$. In particular, $\Flux(\{\phi_t\})(\gamma)$ vanishes on closed curves when $\phi_1$ is Hamiltonian. Actually, any symplectic isotopy with vanishing flux through every closed curve is homotopic to a Hamiltonian isotopy, see \cite[Theorem 10.2.5]{McDuff:1998uu}.

When the isotopy $\{\phi_t\}$ is Hamiltonian with generating Hamiltonian $H_t$, we find the identity
\[
 \Flux(\{\phi_t\})(\gamma) = \int_0^1 H_t(\phi_t(\gamma(1)))\, dt -  \int_0^1 H_t(\phi_t(\gamma(0)))\, dt.
\]
If $\gamma$ is a curve connecting two boundary points, we have
\begin{equation}
\label{flux1}
 \Flux(\{\phi_t\})(\gamma) = \int_0^1 H_t(C_1)\, dt -  \int_0^1 H_t(C_0)\, dt,
\end{equation}
where $C_0$ and $C_1$ are the connected components of $\partial \Sigma$ containing the points $\gamma(0)$ and $\gamma(1)$, respectively, and $H_t(C)$ denotes the common value of $H_t$ on the component $C\subset \partial \Sigma$ (recall that each $H_t$ is constant on every boundary component).

We denote by 
\[
\pi: \widetilde{\Ham}(\Sigma,\omega) \rightarrow \Ham(\Sigma,\omega)
\]
the universal cover of $\Ham(\Sigma,\omega)$. The group $\Ham(\Sigma,\omega)$ is endowed with the $C^1$ topology which is induced by the inclusion in the space of $C^1$ maps from $\Sigma$ to itself. 
The $C^1$ topology on $\Ham(\Sigma,\omega)$ induces a $C^1$ topology on $\widetilde{\Ham}(\Sigma,\omega)$ so that, with respect to these topologies, the covering map $\pi$ is a local homeomorphism.
As usual, we identify the elements of $\widetilde{\Ham}(\Sigma,\omega)$
with homotopy classes with fixed endpoints of Hamiltonian isotopies $\{\phi_t\}$ starting at the identity, so that $\pi([\{\phi_t\}])=\phi_1$.
By the invariance of the flux under homotopies with fixed endpoints of the isotopy and (\ref{flux1}), we deduce that the flux induces a map
\begin{gather*}
\widetilde{\Flux} : \widetilde{\Ham}(\Sigma,\omega) \times H_0(\partial \Sigma)^2 \rightarrow \R, 
\\
\widetilde{\Flux}([\{\phi_t\}],C_0,C_1) = \int_0^1 H_t(C_1)\, dt -  \int_0^1 H_t(C_0)\, dt, 
\end{gather*}
which for any pair $(C_0,C_1)$ restricts to a homomorphism from $\Ham(\Sigma,\omega)$ to $\R$, thanks to the form (\ref{hamprod}) of the Hamiltonian generating the product of two Hamiltonian isotopies. 

\begin{rem}
The above considerations can be restated slightly more abstractly by seeing the flux as a homomorphism from the universal cover of the identity component of the symplectomorphism group of $\Sigma$ to $H^1(\Sigma,\partial \Sigma)$. See \cite[Section 10.2]{McDuff:1998uu} for the case of a closed symplectic manifold.
\hfill\qed
\end{rem}

We shall be particularly interested in the subgroup $\widetilde{\Ham}_0(\Sigma,\omega)$ of $\widetilde\Ham(\Sigma,\omega)$ consisting of Hamiltonian isotopies whose flux through any curve with endpoints on the boundary of $\Sigma$ vanishes, i.e.\
\[
\widetilde{\Ham}_0(\Sigma,\omega) := \Big\{ \tilde{\phi} \in \widetilde{\Ham}(\Sigma,\omega) \ \Big|\  \widetilde{\Flux} \bigl(\tilde{\phi}, C_0,C_1 \bigr) = 0 \quad \forall C_0,C_1 \in H_0(\partial \Sigma) \Big\}.
\]
This is a normal subgroup of $\widetilde{\Ham}(\Sigma,\omega)$ and a proper subgroup whenever $\partial \Sigma$ has more than one connected component. 

\begin{rem}
\label{whenflux0}
 An element $\tilde{\phi}$ of $\widetilde{\Ham}(\Sigma,\omega)$ belongs to $\widetilde{\Ham}_0(\Sigma,\omega)$ if and only if we can normalize the Hamiltonian $H_t$ generating any isotopy $\{\phi_t\}$ representing $\tilde{\phi}$ by requiring
\begin{equation}
\label{norm1}
\int_0^1 H_t(z)\, dt = 0, \qquad \forall z\in \partial \Sigma.
\end{equation}
Similarly, $\tilde{\phi}=[\{\phi_t\}]$ belongs to $\widetilde{\Ham}_0(\Sigma,\omega)$ if and only if we can normalize the action $a$ of $\phi_1$ with respect to any primitive $\nu$ of $\omega$ by requiring
\begin{equation}
\label{norm2}
a(z) = \int_{\{t\mapsto \phi_t(z)\}} \nu, \qquad \forall z\in \partial \Sigma.
\end{equation}
Indeed, (\ref{norm2}) corresponds to the choice $a=a_{H,\nu}$ of (\ref{canaction}), where $H$ is normalized as in (\ref{norm1}).
\hfill\qed
\end{rem}

When $\tilde{\phi}$ belongs to $\widetilde{\Ham}_0(\Sigma,\omega)$ and $\nu$ is a primitive of $\omega$, we shall denote by
\[
a_{\tilde{\phi},\nu} : \Sigma \rightarrow \R
\]
the action of $\pi(\tilde\phi)$ normalized as in~\eqref{norm2}. As the notation suggests, this action does not depend on the choice of the Hamiltonian isotopy representing $\tilde{\phi}$.

If $\tilde{\phi} = [\{\phi_t\}]$ and $\tilde{\psi} = [\{\psi_t\}]$ are in $\widetilde{\Ham}_0(\Sigma,\omega)$ and $\nu$ is any primitive of $\omega$, we have the identity
\begin{equation}
\label{composition}
a_{\tilde{\psi} \circ \tilde{\phi}, \nu} = a_{\tilde{\phi}, \psi_1^* \nu} + a_{\tilde{\psi},\nu}.
\end{equation}
Indeed, one readily checks that the function $a:=a_{\tilde{\phi}, \psi_1^* \nu} + a_{\tilde{\psi},\nu}$ satisfies 
\begin{align*}
&da = (\psi_1 \circ \phi_1)^* \nu - \nu,\\
&a(z) = \int_{\{t\mapsto \psi_t(\phi_t(z))\}}\!\!\!\nu, \qquad \forall z\in \partial \Sigma. 
\end{align*}

A fixed point $z$ of $\tilde{\phi}=[\{\phi_t\}]\in \widetilde{\Ham}(\Sigma,\omega)$ is by definition a fixed point of the map $\phi_1$. Such a fixed point is said to be contractible if the loop $t\mapsto \phi_t(z)$ is contractible in $\Sigma$. The latter condition is clearly independent on the choice of the Hamiltonian isotopy representing $\tilde{\phi}$.

The normalized action $a_{\tilde{\phi},\nu}(z)$ of any contractible fixed point $z$ of $\tilde{\phi}\in \widetilde{\mathrm{Ham}}_0(\Sigma,\omega)$ is independent on the choice of the primitive $\nu$.   Indeed, if $\tilde\phi=[\{\phi_t\}]$ and $H_t$ is the Hamiltonian normalized by \eqref{norm1} generating $\phi_t$, then the identity $a_{\tilde{\phi},\nu}=a_{H,\nu}$ and Stokes' theorem imply
\begin{align}
\label{e:action_nu}
a_{\tilde{\phi},\nu}(z) = \int_{\mathbb{D}} u^* \omega + \int_0^1 H_t(\phi_t(z))\, dt,
\end{align}
where $u: \mathbb{D}\rightarrow \Sigma$ is a capping of the contractible closed curve $t\mapsto \phi_t(z)$. In~\eqref{e:action_nu}, the dependence on $\nu$ disappears. Therefore, we shall denote the normalized action of the contractible fixed point $z$ of $\tilde{\phi}$ simply as $a_{\tilde{\phi}}(z)$. 

Finally, we define the Calabi homomorphism
\[
\Cal: \widetilde{\Ham}_0(\Sigma,\omega) \rightarrow \R,\qquad
\Cal(\tilde{\phi}) := \int_{\Sigma} a_{\tilde{\phi},\nu} \, \omega = 2 \int_0^1\left( \int_{\Sigma} H_t\, \omega \right)\, dt.
\]
The equality of the above two expressions is proven in the lemma below. 
Notice that the above double representation implies that $\Cal(\tilde{\phi})$ is independent of the choice of the primitive of $\nu$ defining the normalized action $a_{\tilde{\phi},\nu}$ and of the choice of the Hamiltonian isotopy representing $\tilde{\phi}$ and defining $H_t$. 
The fact that $\Cal$ is a homomorphism can be proven by either using the representation in terms of action together with (\ref{composition}), or the Hamiltonian representation together with (\ref{hamprod}). 

\begin{lem}
\label{l:Calabi}
For each $\tilde{\phi}=[(\phi_t)]\in\widetilde{\Ham}_0(\Sigma,\omega)$, if $H_t$ is the Hamiltonian normalized by (\ref{norm1}) generating the isotopy $\phi_t$, we have
\begin{align*}
\int_{\Sigma} a_{\tilde{\phi},\nu} \, \omega = 2 \int_0^1\left( \int_{\Sigma} H_t\, \omega \right)\, dt. 
\end{align*}
\end{lem}

\begin{proof}
 From the identity $a_{\tilde{\phi},\nu}=a_{H,\nu}$ we find
\begin{align*}
\int_{\Sigma} a_{\tilde{\phi},\nu} \, \omega 
&= \int_{\Sigma} \left( \int_0^1 \bigl( X_t \lrcorner \, \nu + H_t \bigr) \circ \phi_t\, dt \right) \, \omega = \int_0^1 \left(  \int_{\Sigma}  \bigl( X_t \lrcorner \, \nu + H_t \bigr) \circ \phi_t\, \omega \right) \, dt \\ 
&= \int_0^1 \left(  \int_{\Sigma}  \bigl( X_t \lrcorner \, \nu + H_t \bigr) \, \omega \right) \, dt 
=  \int_0^1 \left(  \int_{\Sigma} \nu \wedge dH_t + H_t\, \omega \right)\, dt,
\end{align*}
where we have used the fact that $\phi_t$ preserves $\omega$, and the identity $(X_t \lrcorner \, \nu)\omega =  \nu \wedge dH_t$. By Stokes theorem, we find
\begin{align*}
\int_0^1 \left(  \int_{\Sigma} \nu \wedge dH_t \right)\, dt & = \int_0^1   \left(  \int_{\Sigma} \bigl( H_t\, d\nu - d(H_t \nu) \bigr) \right)\, dt \\ 
& = \int_0^1\left( \int_{\Sigma} H_t\, \omega \right)\, dt - \int_0^1 \left( \int_{\partial\Sigma} H_t \nu \right)\, dt, 
\end{align*}
and the latter integral vanishes thanks to the normalization condition (\ref{norm1}):
\[
\int_0^1 \left( \int_{\partial\Sigma} H_t \nu \right)\, dt = \sum_{C\in\pi_0(\partial\Sigma)} \left( \int_0^1 H_t(C) \, dt \right) \left( \int_C \nu \right) = 0.
\qedhere
\]
\end{proof}

\subsection{The fixed point theorem.}\label{ss:fixed_point_thm}
We now prescribe the way in which the two-form $\omega$, which is assumed to be symplectic in the interior of $\Sigma$, vanishes on the boundary: 

\begin{ass} 
\label{assumption}
Every connected component $C$ of the boundary $\partial \Sigma$ has a collar neighborhood $A_C\subset \Sigma$ and an identification $A_C\equiv [0,\rho) \times S^1$, for some $\rho>0$ such that
\[
\omega|_{A_C} = -r\, dr\wedge ds.
\]
Here, we are identifying $S^1$ with $\R/\Z$, and $(r,s)$ denotes a point in $[0,\rho) \times S^1$. Note that the orientation of $\partial \Sigma$ as boundary of the oriented surface $(\Sigma,\omega)$ coincides, under the above identification of each component $C\subset \partial \Sigma$ with $\{0\} \times S^1$, with the orientation given by $ds$. 
\hfill\qed
\end{ass}

The main result of this section is the following fixed point theorem, which is stated as Theorem \ref{mt:fixed} in the Introduction and in which we are denoting by
\[
\widehat{\Cal} (\tilde{\phi}) := \frac{\Cal(\tilde{\phi})}{\mathrm{area}(\Sigma,\omega)}
\]
the normalized Calabi invariant of $\tilde\phi\in \widetilde{\Ham}_0(\Sigma,\omega)$. 

\begin{thm}
\label{t:fixed_point}
Assume that the exact two-form $\omega$ on the compact surface $\Sigma$ is symplectic on $\mathrm{int}(\Sigma)$ and satisfies Assumption \ref{assumption}.
For every $c> 0$ there exists a $C^1$-neighborhood $\mathcal{U}$ of the identity in $\widetilde{\Ham}_0(\Sigma,\omega)$ such that every $\tilde{\phi}$ in $\mathcal{U}$ with $\Cal(\tilde{\phi})\leq 0$ has a contractible fixed point $z\in \mathrm{int}(\Sigma)$ whose normalized action satisfies
\begin{equation}
\label{actionbound}
a_{\tilde{\phi}}(z) + c\, a_{\tilde{\phi}}(z)^2 \leq  \frac{1}{2}\,  \widehat{\Cal}(\tilde{\phi}),
\end{equation}
with equality if and only if $\tilde{\phi}$ is the identity.
\end{thm}

In particular, Theorem~\ref{t:fixed_point} implies that any $\tilde{\phi} \in \widetilde{\Ham}_0(\Sigma,\omega)\setminus\{\id\}$ which is sufficiently $C^1$-close to the identity and satisfies $\Cal(\tilde{\phi})\leq 0$ has a contractible interior fixed point $z$ with negative action satisfying 
\begin{align}
\label{e:weaker_fix_point}
a_{\tilde{\phi}}(z) <    \frac{1}{2}\,  \widehat{\Cal}(\tilde{\phi}). 
\end{align}
For the special case when $\Sigma$ is the disk, the weaker conclusion~\eqref{e:weaker_fix_point} is deduced in \cite[Corollary~5]{Abbondandolo:2018fb} from a non-perturbative statement. For arbitrary compact surfaces $\Sigma$ having one boundary component, \eqref{e:weaker_fix_point} is proven in \cite[Corollary~4.16]{Benedetti:2021aa}. The more precise bound which we prove here involving the square of the action  turns out to be important in order to prove systolic inequalities for Reeb flows using quite general global surfaces of section (see Remark \ref{newrem} below for more about this).

\begin{rem}
The upper bound (\ref{actionbound}) can be restated as 
\[
a_{\tilde{\phi}}(z) \leq f_c \big(   \widehat{\Cal}(\tilde{\phi}) \big),
\]
where 
\[
f_c(s) := \frac{1}{2c} \left( \sqrt{1+2cs} - 1 \right) = \frac{1}{2} s - \frac{c}{4} s^2 + O(s^3) \qquad \mbox{for } s\rightarrow 0.
\]
As already observed in \cite[Remark 2.21]{Abbondandolo:2018fb}, the constant $\frac{1}{2}$ in front of the linear term in $s$ is optimal, meaning that it cannot be replaced by a larger constant (recall that the argument of $f_c$ is non-positive): the example that is contained there can be easily modified to produce, for every $\eta> \frac{1}{2}$, an element $\tilde{\phi}$ of $\widetilde{\Ham}_0(\Sigma,\omega)$ which is arbitrarily close to the identity in any $C^k$ norm, has negative Calabi invariant but no contractible fixed point $z$ satisfying
\[
a_{\tilde{\phi}}(z) \leq \eta \cdot   \widehat{\Cal}(\tilde{\phi}).
\]
Therefore, (\ref{e:weaker_fix_point}) can be improved only by considering higher order terms in $s$; the bound (\ref{actionbound}) is such an improvement.
\hfill\qed
\end{rem}

\begin{rem}
By applying Theorem~\ref{t:fixed_point} to $\tilde{\phi}^{-1}$, we obtain the following statement: 
\emph{For every $c> 0$ there exists a $C^1$-neighborhood $\mathcal{U}$ of the identity in $\widetilde{\Ham}_0(\Sigma,\omega)$ such that every $\tilde{\phi}$ in $\mathcal{U}$ with $\Cal(\tilde{\phi})\geq 0$ has a contractible fixed point $z\in \mathrm{int}(\Sigma)$ whose normalized action satisfies}
\begin{equation*}
a_{\tilde{\phi}}(z) - c\, a_{\tilde{\phi}}(z)^2 \geq  \frac{1}{2}  \, \widehat{\Cal}(\tilde{\phi}),
\end{equation*}
\emph{with equality if and only if $\tilde{\phi}$ is the identity.}\hfill\qed
\end{rem}

The proof of Theorem~\ref{t:fixed_point} uses quasi-autonomous Hamiltonians: We recall that the time-dependent Hamiltonian $H_t:\Sigma\to\R$ is called quasi-autonomous if there exist $z_{\min},z_{\max}\in\Sigma$ such that
\[
H_t(z_{\min})=\min_{\Sigma} H_t, \quad H_t(z_{\max})=\max_{\Sigma} H_t, \qquad \forall t\in [0,1].
\]
Note that, if the Hamiltonian isotopy $\{\phi_t\}$ is generated by a quasi-autonomous Hamiltonian $H_t$ as above and $z_{\min}$ and $z_{\max}$ belong to the interior $\mathrm{int}(\Sigma)$, then these points are contractible fixed points of $\tilde{\phi}=[\{\phi_t\}]$. 

Exact symplectomorphisms of $\Sigma$ that are $C^1$-close to the identity are generated by a quasi-autonomous Hamiltonian. More precisely, we have the following result.

\begin{thm}
\label{t:quasi_autonomous}
Assume that the exact two-form $\omega$ on the compact surface $\Sigma$ is symplectic on $\mathrm{int}(\Sigma)$ and satisfies Assumption \ref{assumption}. Let $\phi:\Sigma\to\Sigma$ be an exact symplectomorphism that is sufficiently $C^1$-close to the identity. Then there exists a Hamiltonian isotopy $\{\phi_t\}$ from $\mathrm{id}$ to $\phi$ whose generating Hamiltonian $H_t$ is quasi-autonomous. 
Moreover, for every $\epsilon>0$ there exists $\delta>0$ such that, if $\mathrm{dist}_{C^1}(\phi,\mathrm{id})<\delta$, then:
\begin{enumerate}
\item[$(i)$] $\|H_t\|_{C^1}<\epsilon$ and $\mathrm{dist}_{C^1}(\phi_t,\mathrm{id})<\epsilon$ for every $t\in [0,1]$;
\item[$(ii)$] in the collar neighborhood $A_C\equiv [0,\rho)\times S^1$ of each boundary component $C$ of $\Sigma$ as in Assumption \ref{assumption}, the Hamiltonian $H_t$ has the form
\[
H_t(r,s) = b_C + r^2 h_C(t,r,s)
\]
for some real number $b_C$ and some smooth function $h_C: [0,1] \times A \rightarrow \R$ such that $|b_C|<\epsilon$ and $\|h_C\|_{C^0} < \epsilon$.
\end{enumerate}
\end{thm}

In statements (i) and (ii), the $C^1$ distances and norm are measured with respect to an arbitrary Riemannian metric on $\Sigma$.

\begin{rem}
Theorem \ref{t:quasi_autonomous} has other interesting applications. For instance, it implies that the identity in $\Ham(\Sigma,\omega)$ has a $C^1$-neighborhood on which the Hofer metric is flat. See \cite{Bialy:1994} for more about this in the setting of compactly supported Hamiltonian diffeomorphisms of $\R^{2n}$ and \cite{Lalonde:1995} for the case of compactly supported Hamiltonian diffeomorphisms of more general symplectic manifolds.
\hfill\qed
\end{rem}

This theorem is proven in the next section. Here we will show how the fixed point theorem  can be deduced from it.

\begin{proof}[Proof of Theorem \ref{t:fixed_point}] 
If $\tilde{\phi}$ is the identity, then any point $z\in \mathrm{int}(\Sigma)$ is a contractible fixed point of $\tilde{\phi}$ and $a_{\tilde\phi}(z)=\Cal(\tilde\phi)=0$. Therefore, we must prove that if $\mathcal{U}$ is a sufficiently small $C^1$-neighborhood of the identity in $\widetilde{\Ham}_0(\Sigma,\omega)$ then any $\tilde{\phi}\in\mathcal{U}\setminus\{\id\}$ with $\Cal(\tilde{\phi}) \leq 0$ has a contractible fixed point $z\in \mathrm{int}(\Sigma)$ satisfying the strict inequality  in \eqref{actionbound}. 

We fix
\begin{equation}
\label{chiepsilon}
\epsilon := \frac{N}{4 \area(\Sigma,\omega)\, c}>0,
\end{equation}
where $N\geq 1$ is the number of connected components of $\partial \Sigma$ and $c$ is the arbitrary positive number which appears in the statement we are proving.
By Theorem \ref{t:quasi_autonomous}, if $\mathcal{U}$ is sufficiently small then any
$\tilde{\phi}\in\mathcal{U}\setminus\{\id\}$ is represented by a Hamiltonian isotopy $\{\phi_t\}$ which is generated by a quasi-autonomous Hamiltonian $H_t$ satisfying the bounds (i) 
and (ii) for the $\epsilon$ given by~\eqref{chiepsilon}.

Since $\tilde{\phi}$ belongs to $\widetilde{\Ham}_0(\Sigma,\omega)$ and $H$ is constant on $[0,1]\times C$ for every connected component $C$ of $\partial \Sigma$, up to adding a suitable constant we may assume that $H_t$ vanishes on $\partial \Sigma$ for every $t\in [0,1]$. By Theorem \ref{t:quasi_autonomous}(ii), on the collar neighborhood $A_C \equiv [0,\rho) \times S^1$ of every connected component $C$ of $\partial \Sigma$ the Hamiltonian $H_t$ has the form
\begin{align}
\label{specfor}
H_t(r,s) = r^2 h_C(t,r,s), \qquad \mbox{where} \quad \|h_C\|_{C^0} < \epsilon.
\end{align}
Since $H_t$ is quasi-autonomous, there exists $z_{\min}\in \Sigma$ which minimizes $H_t$ for every $t\in [0,1]$. Since $H_t$ vanishes on $\partial \Sigma$, we have $H_t(z_{\min}) \leq 0$ for every $t\in [0,1]$. Since
\[
\int_0^1 \left( \int_{\Sigma} H_t \, \omega \right)\, dt = \frac{1}{2} \Cal(\tilde{\phi}) \leq 0,
\]
and since $H$ does not vanish identically, because $\tilde{\phi} \neq \mathrm{id}$, $H_t(z_{\min})$ is strictly negative for some $t\in[0,1]$. In particular, $z_{\min}$ belongs to the interior of $\Sigma$ and hence is a contractible fixed point of $\tilde{\phi}$ of action 
\[
a_{\tilde{\phi}}(z_{\min}) = \int_0^1 H_t(z_{\min})\, dt<0.
\]
In order to estimate this action, we introduce the function
\[
K: \Sigma\rightarrow \R, \qquad K(z):= \int_0^1 H_t(z)\, dt.
\]
The point $z_{\min}$ minimizes $K$, and 
\[
-m:= K(z_{\min}) = a_{\tilde{\phi}}(z_{\min})<0. 
\]
Consider the collar neighborhood $A_C\equiv [0,\rho) \times S^1$ of some connected component $C$ of $\partial \Sigma$. By (\ref{specfor}), we have
\[
|K(r,s)| \leq \epsilon r^2, \qquad \forall (r,s)\in A_C.
\]
Together with the fact that $K\geq -m$ on $\Sigma$, we deduce that
\begin{equation}
\label{lower}
K(r,s) \geq \max\{ -\epsilon r^2 , -m\} 
\qquad \forall (r,s)\in A_C.
\end{equation}
Up to reducing if necessary the neighborhood $\mathcal{U}$, we can make the $C^0$-norm of $H$ as small as we wish and hence we may assume that $m< \epsilon \rho^2$. Therefore,
\[
\max\{ -\epsilon r^2 , -m\}  = 
\left\{ 
\begin{array}{@{}ll} 
- \epsilon r^2 , & \forall r\in \Bigl[ 0,\sqrt{m/\epsilon} \bigr],\vspace{5pt}\\ 
-m, & \forall r\in \bigl[ \sqrt{m/\epsilon}, \rho \bigr). 
\end{array} 
\right.
\]
By integrating \eqref{lower} over $A_C$, we infer
\[
\begin{split}
\int_{A_C} K\, \omega 
&> 
\int_{[0,\rho) \times S^1} \max\{ -\epsilon r^2 , -m\} r\, dr\wedge ds 
= \int_0^{\rho} \max\{ -\epsilon r^2 , -m\} r\, dr \\ 
&= - \epsilon \int_0^{\sqrt{m/\epsilon}}r^3\, dr - m \int_{\sqrt{m/\epsilon}}^{\rho} r\, dr = -\frac{m}{2} \rho^2 + \frac{m^2}{4\epsilon}\\
& 
= 
-m\area(A_C,\omega) + \frac{m^2}{4\epsilon} , 
\end{split}
\]
Note that we have written a strict inequality here because the inequality in (\ref{lower}) cannot be everywhere an equality, as the right-hand side is not a differentiable function of $r$ at $r=\sqrt{m/\epsilon} \in (0,\rho)$. On the other hand, on $\Sigma\setminus U$, where $U$ denotes the union of the collar neighborhoods $A_C$ of the $N$ components of $\partial \Sigma$ we have
\[
\int_{\Sigma\setminus U} K\, \omega \geq -m \area(\Sigma\setminus U,\omega).
\]
Putting these inequalities together, we find
\begin{align}
\label{e:final_fixpoint_ineq}
\int_{\Sigma} K\, \omega > -m\area(\Sigma,\omega) + N \frac{m^2}{4\epsilon}. 
\end{align}
Since $-m=a_{\tilde{\phi}}(z_{\min})$ and the integral of $K$ on $\Sigma$ is $\tfrac12\Cal(\tilde{\phi})$, the inequality~\eqref{e:final_fixpoint_ineq} and our choice of $\epsilon$ in (\ref{chiepsilon}) give us the desired conclusion:
\[
a_{\tilde{\phi}}(z_{\min}) + c\, a_{\tilde{\phi}}(z_{\min})^2
<  \frac{\Cal(\tilde{\phi})}{ 2\, \area(\Sigma,\omega)}  = \frac{1}{2} \, \widehat{\Cal}(\tilde{\phi}).
\qedhere
\]
\end{proof}

\subsection{Construction of a generating quasi-autonomous Hamiltonian}\label{ss:quasi_autonomous}

The aim of this section is to prove Theorem \ref{t:quasi_autonomous}. The proof closely follows the argument of \cite[Section 4]{Benedetti:2021aa}, but for sake of completeness we work out the details.

By Assumption \ref{assumption} and up to reducing the positive number $\rho$ appearing there, we can find a primitive $\nu_0$ of $\omega$ which on the collar neighborhood $A_C\equiv [0,\rho)\times S^1$ of each component $C$ of the boundary of $\Sigma$ has the form
\begin{equation}
\label{e:lambda|_A}
\nu_0|_{A_C} = \bigl( a_C - \tfrac{1}{2} r^2 \bigr)\, ds,
\end{equation}
for some $a_C\in \R$. Indeed, if $\nu$ is any primitive of $\omega$ and $a_C$ is its integral on the boundary component $C = \{0\} \times S^1$, then the one-form above differs from $\nu|_{A_C}$ by the differential of a function $g_C$. By adding to $\nu$ the differential of a function on $\Sigma$ that agrees with $g_C$ on a slighltly reduced collar neighborhood of every component $C$ of $\partial \Sigma$, we obtain the desired primitive $\nu_0$.

For $i=1,2$, we consider the one-forms $\nu_i:=\pr_i^*\nu_0$ on $\Sigma\times\Sigma$, where 
\[
\pr_i:\Sigma\times\Sigma\to\Sigma,  \qquad \pr_i(z_1,z_2)=z_i,
\]
and the standard Liouville form $\lambda_{\mathrm{std}}$ on the cotangent bundle $\Tan^*\Sigma$, which is uniquely defined by the equation $\alpha^*\lambda_{\mathrm{std}}=\alpha$ for all one-forms $\alpha$ on $\Sigma$.

If $z_1$ and $z_2$ are points on the same connected component $C$ of $\partial \Sigma$, then they are identified with pairs $(0,s_1)$, $(0,s_2)$ in $A_C=[0,\rho)\times S^1$. When $s_2$ and $s_1$ are not antipodal in $S^1=\R/\Z$, meaning that $|s_2-s_1|< \tfrac 12$ for suitable lifts to $\R$, we denote by $[z_1,z_2]$ the unique shortest oriented arc in $C$ from $z_1$ to $z_2$, so that
\begin{align*}
\int_{[z_1,z_2]} \!\!\! \diff s < \tfrac12.
\end{align*}

The next result is a version of Weinstein tubular neighborhood theorem in our setting.

\begin{lem}[Weinstein tubular neighborhood]
\label{l:Weinstein}
There exists an open neighborhood $U\subset\Sigma\times\Sigma$ of the diagonal $\Delta_\Sigma=\{(z,z)\ |\ z\in\Sigma\}$, an open neighborhood $V\subset\Tan^*\Sigma$ of the $0$-section $0_\Sigma\subset\Tan^*\Sigma$, and a smooth map $\psi:U\to V$ that restricts to a diffeomorphism $\psi:U\cap\mathrm{int}(\Sigma\times\Sigma)\to V\cap \Tan^* \mathrm{int} (\Sigma)$, and satisfies
\[
\psi(\Delta_\Sigma)=0_\Sigma,
\qquad
\psi^*\lambda_{\std}=\nu_2-\nu_1 + \diff f,
\] where $f:\Sigma\times\Sigma\to\R$ is a smooth function such that $f|_{\Delta_\Sigma}\equiv0$ and 
\begin{align*}
 \diff f(z_1,z_2)=(\nu_1-\nu_2)_{(z_1,z_2)},
 \qquad
 f(z_1,z_2)=-\int_{[z_1,z_2]} \!\!\!\!\!\!\nu_0,
\end{align*}
for every pair $(z_1,z_2)\in U\cap(\partial\Sigma\times\partial\Sigma)$.
If  $A_C\equiv[0,\rho)\times S^1$ is the collar neighborhood of a connected component $C$ of $\partial\Sigma$  on which $\nu_0$ has the form (\ref{e:lambda|_A}), the restriction $\psi|_{U\cap(A_C\times A_C)}$ has the form
\begin{equation}
\label{psi}
\begin{split}
 \psi(r,s,R,S)=\big( R,s,\,R\,(S-s),\tfrac 1 2 (r^2-R^2) \big),
 \\
 \forall (r,s,R,S)\in U\cap(A_C\times A_C).
\end{split} 
\end{equation}
\end{lem}

\begin{proof}
We first provide the construction within the collar neighborhood $A=A_C$ of each connected component $C$ of $\partial\Sigma$.
We consider a small enough neighborhood $W\subset A\times A$ of the diagonal $\Delta_A=\{(z,z)\ |\ z\in A\}$ so that, for each $(r,s,R,S)\in W$, the points $s,S\in S^1$ are not antipodal. We define the map 
\begin{align*}
 \kappa_0:W\to \underbrace{\Tan^*A}_{A\times \R^2},
 \qquad
 \kappa_0(r,s,R,S)=\big( R,s,\,R\,(S-s),\tfrac 12 (r^2-R^2) \big).
\end{align*}
This map restricts to a diffeomorphism onto its image
\[\kappa_0:W\cap\mathrm{int}(A\times A)\to \mathrm{int}(\Tan^*A),\] 
and satisfies $\kappa_0(\Delta_A)=0_A$ and
\begin{align*}
\kappa_0^*\lambda_{\std}
&=
R\,(S-s)\,\diff R + \tfrac 12(r^2-R^2)\,\diff s\\
&=
-\tfrac 12\,R^2\,\diff S + \tfrac 12\,r^2\,\diff s + \diff\big( \tfrac 12(S-s)\,R^2 \big)\\
&=
-\nu_1+\nu_2+\diff\big( (a_C -\tfrac 12 R^2)(s-S) \big)\\
&=-\nu_1+\nu_2+\diff f_0,
\end{align*}
where
\begin{align*}
 f_0:W\to\R,\qquad f_0(r,s,R,S)=(a_C -\tfrac 12 R^2)(s-S).
\end{align*}
Notice that $f_0|_{\Delta_A}\equiv0$ and 
\begin{align*}
f_0(0,s,0,S) & =a_C (s-S) = - \int_{[z_1,z_2]} \nu_0,
\\
\diff f_0(0,s,0,S) & = a_C \,\diff s-a_C\,\diff S = \nu_1-\nu_2,
\end{align*}
where $z_1=(0,s)$ and $z_2=(0,S)$.

For some sufficiently small neighborhood $U\subset\Sigma\times\Sigma$ of the diagonal $\Delta_\Sigma$, we choose an arbitrary smooth function $f_1:U\to\R$ that coincides with $f_0$ on a neighborhood of $U\cap(\partial\Sigma\times\partial\Sigma)$ and satisfies $f_1|_{\Delta_\Sigma}\equiv0$ and $\diff f_1=\nu_1-\nu_2$ at all points of the diagonal $\Delta_\Sigma$. Up to further shrinking the neighborhood $U$, we also choose a smooth map $\kappa_1:U\to\Tan^*\Sigma$ that coincides with $\kappa_0$ on a neighborhood of $W\cap(\partial\Sigma\times\partial\Sigma)$, restricts to a diffeomorphism onto its image $\kappa_1:U\cap\mathrm{int}(\Sigma\times\Sigma)\to\Tan^*(\mathrm{int}(\Sigma))$, and such that $\kappa_1(\Delta_\Sigma)=0_\Sigma$.

We now conclude the proof by means of a typical application of Moser's trick. We set 
\begin{align*}
 \mu_t:=t\kappa_1^*(\lambda_{\std}) + (1-t)(\nu_2-\nu_1),
\end{align*}
and we look for an isotopy $\phi_t:U\to\Sigma\times\Sigma$, defined after possibly further shrinking the neighborhood $U$, such that $\phi_0=\id$ and $\phi_t^*\mu_t-\mu_0$ is exact. We denote by $X_t$ the time-dependent vector field generating $\phi_t$ and compute
\begin{align*}
\tfrac{\diff}{\diff t}\phi_t^*\mu_t
&=\phi_t^* ( L_{X_t} \mu_t +  \tfrac{\diff}{\diff t} \mu_t) = \phi_t^*\big( X_t\,\lrcorner\,\diff\mu_t + \diff(\mu_t(X_t))+ \mu_1-\mu_0 \big)\\
&=\phi_t^*\big( X_t\,\lrcorner\,\diff\mu_t + \kappa_1^*\lambda_{\std}-\nu_2+\nu_1-\diff f_1 + \diff(\mu_t(X_t)+f_1)\big).
\end{align*}
Notice that the symplectic forms $\kappa_1^*\diff\lambda_{\std}$ and $-\diff\nu_1+\diff\nu_2$ define the same orientation, since they coincide on a neighborhood of $U\cap(\partial\Sigma\times\partial\Sigma)$. Therefore $\diff\mu_t$ is symplectic away from $\partial(\Sigma\times \Sigma)$ for every $t\in[0,1]$. We choose the vector field $X_t$ so that
\begin{align*}
X_t\,\lrcorner\,\diff\mu_t + \kappa_1^*\lambda_{\std}-\nu_2+\nu_1-\diff f_1 
=
0.
\end{align*}
Notice that $X_t$ vanishes on the diagonal $\Delta_\Sigma$ and on a neighborhood of $U\cap(\partial\Sigma\times\partial\Sigma)$. 
Moreover
\begin{align*}
 \tfrac{\diff}{\diff t}\phi_t^*\mu_t
 =
 \phi_t^*\diff(\mu_t(X_t)+f_1).
\end{align*}
Up to shrinking the neighborhood $U$, we obtain a well defined isotopy $\phi_t:U\to\Tan^*\Sigma$ that coincides with $\kappa_0$ on a neighborhood of $U\cap(\partial\Sigma\times\partial\Sigma)$ and satisfies
\begin{align*}
\phi_1^*\kappa_1^*(\lambda_{\std})=-\nu_1+\nu_2 + \diff f,
\end{align*}
where
\begin{align*}
 f(w)=\int_0^1 \big(\mu_t(X_t)+f_1\big)\circ\phi_t(w)\,\diff t.
\end{align*}
The desired map is $\psi:=\kappa_1\circ\phi_1$.
\end{proof}

\begin{proof}[Proof of Theorem~\ref{t:quasi_autonomous}] 
We still work with the special primitive $\nu_0$ of $\omega$ satisfying (\ref{e:lambda|_A}).
By assumption, the diffeomorphism $\phi: \Sigma \rightarrow \Sigma$ satisfies
\[
\phi^* \nu_0 - \nu_0 = da
\]
for some smooth function $a$ on $\Sigma$. Note that $a$ is $C^1$-small when $\phi$ is $C^1$-close to the identity. We consider the associated map
\[
\Phi: \Sigma \rightarrow \Sigma \times \Sigma, \qquad \Phi(z) = (z,\phi(z)).
\]
We require $\phi$ to be sufficiently $C^1$-close to the identity so that the image of $\Phi$ is contained in the domain of the map $\psi:U\to\Tan^*\Sigma$ provided by Lemma~\ref{l:Weinstein}, and the image of $\psi\circ\Phi$ is a section of the cotangent bundle $\Tan^*\Sigma$. Namely, if we denote by $\pi:\Tan^*\Sigma\to\Sigma$ the projection onto the base of the cotangent bundle, the map 
\begin{align*}
\widetilde\phi:\Sigma\to\Sigma,
\qquad
\widetilde\phi(z)=\pi\circ\psi\circ\Phi(z)
\end{align*}
is a diffeomorphism.
We consider the smooth function $f:\Sigma\times\Sigma\to\R$ provided by Lemma~\ref{l:Weinstein}. 
Since
\begin{align*}
 (\psi\circ\Phi)^*\lambda_{\std} 
 = 
 \Phi^*(\nu_2-\nu_1+\diff f)
 =
 \phi^*\nu-\nu+\diff(f\circ\Phi)
 =
 \diff( a+f\circ\Phi),
\end{align*}
we have that 
\begin{align}
\label{defF}
\psi\circ\Phi(z)=(\widetilde\phi(z),\diff F(\widetilde\phi(z))),
\end{align}
where $F:\Sigma\to\R$ is the smooth generating function
\begin{align*}
 F(w)=( a+f\circ\Phi)\circ\widetilde\phi^{-1}(w).
\end{align*}
Identity (\ref{defF}) implies that $F$ is $C^2$-small when $\phi$ is $C^1$-close to the identity.

Consider now the collar neighborhood $A=A_C\equiv[0,\rho)\times S^1$ of a connected component $C$ of $\partial\Sigma$ as in (\ref{e:lambda|_A}). For all $(r,s)\in A\cap\phi^{-1}(A)$, if we set $(R,S)=\phi(r,s)$, we have $\widetilde\phi(r,s)=(R,s)$ and
\begin{equation}
\label{genfun}
 R\,(S-s) = \partial_R F(R,s),\qquad
 \tfrac 12 (r^2-R^2) = \partial_s F(R,s).
\end{equation}
This implies that $\diff F=0$ at all points of $\partial\Sigma$. In particular, $F$ is constant on each component $C$ of $\partial \Sigma$ and in a neighborhood of this component we can write $F$ as
\begin{equation}
\label{Fnearbdry}
 F(R,s)=b+ R^2G(R,s),
\end{equation}
where $b=b_C$ is a real number and $G=G_C$ is a smooth function. More precisely, 
\[
b= \lim_{R\rightarrow 0} \frac{\partial_R F(R,s)}{R} =  S(0,s) - s
\]
and the Taylor theorem with integral remainder gives us the formula
\[
F(R,s) = b + R^2 \int_0^1 \partial_R^2 F(tR,s) (1-t)\, dt.
\] 
By differentiating the first identity in (\ref{genfun}) with respect to $R$, we find
\[
G(R,s) = \int_0^1 \partial_R^2 F(tR,s) (1-t)\, dt = \int_0^1 \bigl( -s + S(tR,s) + tR \,\partial_R S(tR,s)  \bigr) (1-t)\, dt.
\]
The above formulas for $b$ and $G$ imply that $|b|$ and $\|G\|_{C^0}$ are both small if $\phi$ is $C^1$-close to the identity.

We now consider the isotopy 
\[\phi_t:\mathrm{int}(\Sigma)\to\mathrm{int}(\Sigma),\ t\in[0,1],\] 
with the associated map $\Phi_t:\mathrm{int}(\Sigma)\to\mathrm{int}(\Sigma\times\Sigma)$, $\Phi_t(z)=(z,\phi_t(z))$, whose image $\psi\circ\Phi_t(\mathrm{int}(\Sigma))$ is the graph of $t\,\diff F$. Notice that $\phi_t$ defines an associated diffeomorphism 
\[
\widetilde\phi_t:=\pi\circ\psi\circ\Phi_t:\mathrm{int}(\Sigma)\to\mathrm{int}(\Sigma),
\]
and
\begin{align}
\label{e:graph_tdF}
\psi\circ\Phi_t(z)=(\widetilde\phi_t(z),t\,\diff F(\widetilde\phi_t(z))).
\end{align}
The endpoints of the isotopy are $\phi_0=\id$ and $\phi_1=\phi$. We claim that $\phi_t$ extends as a smooth isotopy $\phi_t:\Sigma\to\Sigma$ that is $C^1$-close to the identity. In order to prove this, let us focus on the collar neighborhood $A_C\equiv[0,\rho)\times S^1$ of a connected component $C$ of $\partial\Sigma$.  If we write $(R_t,S_t):=\phi_t(r,s)$, then Equation~\eqref{e:graph_tdF} in the annulus $\mathrm{int}(A_C)$ becomes
\begin{align*}
R_t\,(S_t-s)  = t\,\partial_RF(R_t,s),\qquad
\tfrac 12 (r^2-R_t^2)  = t\,\partial_sF(R_t,s),
\end{align*}
that is, using (\ref{Fnearbdry}),
\begin{align*}
S_t  = s + \underbrace{\big(2 G(R_t,s) +  R_t\,\partial_R G(R_t,s) \big) t}_{(*)},\qquad
r  = R_t\underbrace{\sqrt{ 1 + 2 t\,\partial_sG(R_t,s) }}_{(**)}.
\end{align*}
The term $(*)$ is $C^1$-small and the term $(**)$ is $C^1$-close to 1  as functions of $(R_t,s)$. This shows that the isotopy $(R_t,s)\mapsto(r,S_t)$ is $C^1$-close to the identity and extends smoothly to the boundary $C$ by $(0,s)\mapsto(0,s +2 t\,G(0,s))$. Therefore, we obtain a $C^1$-close to the identity smooth extension $\phi_t:\Sigma\to\Sigma$ as well.

Let $X_t$ bt the time dependent vector field generating the isotopy $\phi_t$. We claim that $X_t$ is Hamiltonian with Hamiltonian function
\begin{align}
\label{H_t}
H_t:\Sigma\to\R,\qquad H_t(z) := F\circ\pi\circ\psi(\phi_t^{-1}(z),z).
\end{align}
Indeed, consider an arbitrary $v\in\Tan_z\Sigma$ and  set
\[w:=\diff\widetilde\phi_t(z)v, \qquad q_t:=\widetilde\phi_t(z), \qquad y_t:=\tfrac{\diff}{\diff t}\widetilde\phi_t(z).\] 
In Darboux coordinates, we locally see $\Tan^*\Sigma$ as $\Sigma\times\R^2$, and compute
\begin{align*}
 \diff\nu\big(X_t(\phi_t(z)),\diff\phi_t(z)v\big)
 &=
 \psi^*\diff\lambda_{\std} \big( (0,X_t(\phi_t(z))) , (v,\diff\phi_t(z)v) \big)\\
 &=
 \diff\lambda_{\std}\big( (y_t,\diff F(q_t) + t\,\diff^2F(q_t)y_t) , (w,t\,\diff^2F(q_t)w) \big)\\
 &=
 \diff F(q_t)w + t\,\diff^2F(q_t)[y_t,w] - t\,\diff^2F(q_t)[w,y_t]\\
 &=
 \diff F(q_t)w
 = \diff (F\circ\widetilde\phi_t)(z)v\\
 &= \diff (F\circ\widetilde\phi_t\circ\phi_t^{-1})(\phi_t(z))\diff\phi_t(z)v\\
 &=\diff H_t(\phi_t(z))\diff\phi_t(z)v,
\end{align*}
proving our claim.

Note that for every $t\in [0,1]$, the maximum and minimum of $H_t$ on $\Sigma$ coincide with those of $F$. Note also that by (\ref{Fnearbdry}) we have
\begin{equation}
\label{sopra}
H_t(z) = F(z) = b_C \qquad \forall z\in C,
\end{equation}
for every connected component $C$ of $\partial \Sigma$. Moreover, the previous considerations on $F$, $b_C$ and $G_C$ imply that if $\phi$ is $C^1$-close to the identity, then $H_t$ is $C^1$-small and on $A_C$ has the form
\[
H_t(r,s) = b_C + r^2 h_C(t,r,s),
\]
where both $|b_C|$ and $\|h_C\|_{C^0}$ are small. Indeed, the above identity and the $C^0$-smallness of $h_C$ follow from (\ref{psi}) and (\ref{H_t}), which give us the identity
\[
H_t(r,s) = F(r,\overline{S}_t(r,s)) = b_C + r^2 G_C(r,\overline{S}_t(r,s))
\]
where $\phi_t^{-1}(r,s)) = ( \overline{R}_t(r,s),\overline{S}_t(r,s))$. Together with the already mentioned $C^1$-closeness of $\phi_t$ to the identity, this proves statements (i) and (ii) in Theorem \ref{t:quasi_autonomous}. 

Let us check that for every $t\in [0,1]$ the function $H_t$ achieves its minimum at some point $z_{\min}\in \Sigma$ which is independent of $t$. If $F$ achieves its minimum on $\partial \Sigma$, this follows from the identity $\min H_t = \min F$ and (\ref{sopra}). So let us assume that $F$ achieves its minimum at an interior point $q_{\min}$. Then $dF(q_{\min})=0$ and, since the inverse image of the zero-section in $T^* \mathrm{int}(\Sigma)$ is the diagonal in $\mathrm{int}(\Sigma) \times \mathrm{int}(\Sigma)$, we have
\[
\psi(z_{\min},z_{\min}) = (q_{\min},0)
\]
for some $z_{\min} \in \mathrm{int}(\Sigma)$. Since $\psi$ maps the graph of $\phi_t$ to the graph of $t\, dF$, we have
\[
\phi_t(z_{\min}) = z_{\min} \qquad \forall t\in [0,1],
\]
and hence
\[
H_t(z_{\min}) = F \circ \pi \circ \psi(z_{\min},z_{\min}) = F \circ \pi (q_{\min},0) = F(q_{\min}).
\] 
This shows that
\[
H_t(z_{\min}) = \min_{\Sigma} H_t \qquad \forall t\in [0,1].
\]
The argument for the maximum is analogous, and we conclude that $H_t$ is quasi-autonomous.
\end{proof}

\section{Global surfaces of section for nearly Besse contact forms on 3-manifolds}
\label{s:sos}

\subsection{Global surfaces of section} 
\label{s:sos_def}
In this paper, a global surface of section for a contact form $\lambda$ on a 3-manifold $Y$ is a smooth map $\iota:\Sigma\to Y$, where $\Sigma$ is an oriented connected compact surface with non-empty and possibly disconnected boundary, with the following properties:
\begin{itemize}

\item \textit{$($Boundary$\,)$} The restriction $\iota|_{\partial\Sigma}$ is an immersion positively tangent to the Reeb vector field $R_{\lambda}$. Namely, with the orientation on the boundary $\partial\Sigma$ induced by the one of $\Sigma$,
the restriction of $\iota$ to any connected component of $\partial\Sigma$ is an orientation preserving covering map of a closed Reeb orbit of $(Y,\lambda)$.\vspace{2pt}

\item \textit{$($Transversality$\,)$} The restriction $\iota|_{\mathrm{int}(\Sigma)}$ is an embedding into $Y\setminus\iota(\partial\Sigma)$ transverse to the Reeb vector field $R_\lambda$. 
In particular, the 2-form $\iota^*\diff\lambda$ is nowhere vanishing on $\mathrm{int}(\Sigma)$, and we assume that it is a positive area form on the oriented surface $\mathrm{int}(\Sigma)$.\vspace{2pt}

\item \textit{$($Globality$\,)$} For each point $z\in Y$, the Reeb orbit $t\mapsto\phi_\lambda^t(z)$ intersects $\Sigma$ in both positive and negative time.

\end{itemize}
We stress that, in the literature, the notion of global surface of section may be slightly different than the one given here: for instance, the map $\iota:\Sigma\to Y$ may be required to be an embedding, or the restriction of $\iota$ to some connected component of $\Sigma$ may be allowed to be an orientation reversing covering map of a closed Reeb orbit.

\subsection{The surgery description of a Besse contact form on a 3-manifold}
\label{s:surgery}

The proof of Theorem~\ref{t:main} will require suitable surfaces of section for the Reeb flows of contact forms sufficiently $C^3$-close to a Besse one. As a preliminary step, in the next subsection we shall construct global surfaces of sections for the Reeb flow of Besse contact 3-manifolds. It will be useful to employ the surgery description of Besse contact 3-manifolds as Seifert fibered spaces, which we now recall. We refer the reader to, e.g., \cite{Orlik:1972ts,Jankins:1983zm} for more details.

A Seifert fibration $\pi:Y\to B$, in the generality that we need for the study of Besse contact 3-manifolds, is defined up to a suitable notion of isomorphism by a genus and $k\geq1$ pairs of coprime integers $(\alpha_1,\beta_1),...,(\alpha_k,\beta_k)\in\N\times\Z$. Here, $\N$ denotes the set of positive integers and the coprimeness assumption implies that $\alpha_j=1$ if $\beta_j=0$. We denote by $B_0$ an oriented compact connected surface of the given genus with $k$ boundary components. We write its oriented boundary as $\partial B_0=\partial_1 B_0\cup...\cup\partial_{k} B_0$, where each $\partial_j B_0$ is a connected component oriented as the boundary of $B_0$. Over $B_0$, we consider the trivial $S^1$-bundle 
\[\pi:Y_0:=B_0\times S^1\to B_0,\qquad \pi(z,t)=z,\] 
with its associated free $S^1$-action 
\begin{align*}
 t\cdot(z,s)=(z,s+t),\qquad \forall t\in S^1,\ (z,s)\in Y_0.
\end{align*}
Here and elsewhere in the paper, $S^1=\R/\Z$.
Next, we consider $k$ disjoint copies $B_j\subset \C$, $j=1,...,k$ of the disk of some radius $\rho>0$ centered at the origin, and the solid tori together with their base projections 
\begin{align*}
 \pi:Y_j:=B_i\times S^1\to B_j,\qquad \pi(z,t)=z.
\end{align*}
By the B\'ezout identity, we can find pairs of coprime integers $(\alpha_j',\beta_j')\in \Z\times \Z$ such that 
\begin{align}
\label{e:dual_Seifert_pair}
\det\left(
  \begin{array}{@{}cc@{}}
    \alpha_j & \alpha_j' \\ 
    \beta_j & \beta_j' 
  \end{array}
\right)=1,
\qquad
\forall j=1,...,k.
\end{align}
The pair $(\alpha_j',\beta_j')$ is not uniquely determined by $(\alpha_j,\beta_j)$ (except if $(\alpha_j,\beta_j)=(1,0)$, in which we necessarily have $(\alpha_j',\beta_j') = (0,1)$). We introduce the oriented curves
\begin{align*}
m_j&=\partial B_j\times\{*\}\subset \partial Y_j,
& 
l_j&=\{*\}\times S^1  \subset \partial Y_j,
\\
h_j&=\{*\}\times S^1\subset \partial_j B_0\times S^1 \subset\partial Y_0,
&
 f_j&=-\partial_j B_0\times\{*\}  \subset\partial Y_0.
\end{align*}
Here, we used the symbol $*$ to denote an arbitrary point of a space. We glue $Y_0,Y_1,...,Y_k$ along their boundaries by identifying
\begin{align*}
 m_j\equiv\alpha_jf_j+\beta_j h_j,\qquad l_j\equiv\alpha_j'f_j+\beta_j'h_j,\qquad \forall j=1,...,k
\end{align*}
and denote by $Y$ the resulting closed 3-manifold. Here, we mean that $m_j\subset\partial Y_j$ is identified with an oriented embedded circle in $\partial_j Y_0$ that is homologous to $\alpha_j[f_j]+\beta_j[h_j]$, and analogously for $l_j$.
The free $S^1$-action on $Y_0$ extends to an $S^1$-action on the whole $Y$, which on the solid tori $Y_j$ has the form
\begin{align*}
t\cdot(z,s)=\big(z e^{-2\pi\alpha_j't i},s+\alpha_j t\big),
\qquad
\forall t\in S^1,\ (z,s)\in Y_j.
\end{align*}
If $\alpha_j>1$, the $S^1$-action is not free on the orbit 
\[\gamma_j:=\{0\}\times S^1\subset Y_j,\] 
and in this case we call such an orbit singular. All those $S^1$-orbits that are not singular are called regular.
The surfaces $B_j$ are glued accordingly to form a closed surface $B$. The maps $\pi$ are glued as well, and the obtained $\pi:Y\to B$ is the quotient projection of the $S^1$-action on $Y$.
We can always assume without loss of generality that the number of Seifert pairs $(\alpha_j,\beta_j)$ is $k\geq2$; indeed, adding the trivial Seifert pair $(1,0)$ does not affect the Seifert fibration. 

The locally free $S^1$-action that is induced by a Besse Reeb flow on $Y$ can be seen as the $S^1$-action on the total space of a Seifert fibration $\pi:Y\rightarrow B$ as above. In this case, the Euler number
\begin{align*}
\eul(Y):=-\frac{\beta_1}{\alpha_1}-...-\frac{\beta_k}{\alpha_k}
\end{align*}
is negative, as shown in \cite{Lisca:2004oz} (see also \cite[Theorem 1.4]{Kegel:2021vo}).

\subsection{A Global surface of section for a Besse contact form on a 3-manifold}
\label{s:sos_Besse}
The existence of global surfaces of section in Seifert spaces was thoroughly investigated by  Albach-Geiges \cite{Albach:2021wi}. In this subsection, we prove a statement which may be of independent interest (Theorem~\ref{t:sos_Besse}) asserting that on a Besse contact 3-manifold any given closed Reeb orbit is the (multiply covered) boundary of a surface of section as defined in Subsection~\ref{s:sos_def}

Let $(Y,\lambda_0)$ be a Besse contact 3-manifold whose Reeb orbits have minimal common period 1, and $\gamma_1$ be an arbitrary closed Reeb orbit. The Reeb flow $\phi_{\lambda_0}^t$ defines a locally free $S^1$-action on $Y$. We can see this action as the $S^1$-action of a Seifert fibered structure which we describe with the same notation of the previous subsection: in particular, for each $j=1,...,k$, we denote by $\gamma_j$ the closed Reeb orbit corresponding to the Seifert pair $(\alpha_j,\beta_j)$. Notice in particular that we are assuming without loss of generality that our given $\gamma_1$ is the closed Reeb orbit corresponding to the Seifert pair $(\alpha_1,\beta_1)$. We recall that $\gamma_1$ has minimal period $1/\alpha_1$, and therefore it is a singular orbit if and only if $\alpha_1>1$.

We consider the tubular neighborhood $Y_1\subset Y$ of $\gamma_1$, which was realized as $Y_1\equiv B_1\times S^1$, and under this identification we have $\gamma_1=\{0\}\times S^1\subset Y_1$.
We equip $Y_1$ with the contact form
\begin{align*}
\lambda_0' = \tfrac{1}{\alpha_1}r^2\diff\theta + \tfrac{1}{\alpha_1}\big(1+2\pi\tfrac{\alpha'_1}{\alpha_1}r^2\big) \diff s,
\end{align*}
where $(r,\theta)$ are polar coordinates on the disk $B_1$, and $s\in S^1=\R/\Z$.
The associated Reeb vector field is given by
\[R_{\lambda_0'} = -2\pi\alpha_1'\partial_\theta + \alpha_1\partial_s=R_{\lambda_0}|_{Y_1},
\]
and therefore the associated Reeb flow agrees with the Seifert $S^1$-action
\begin{align*}
 \phi_{\lambda_0'}^t(z,s)=\phi_{\lambda_0}^t(z,s)=t\cdot(z,s)=\big(z e^{-2\pi\alpha_1't i},s+\alpha_1 t\big),
\end{align*}
Up to pulling back $\lambda_0$ by an $S^1$-equivariant diffeomorphism, we can assume that 
\[\lambda_0|_{Y_1}=\lambda_0'=\tfrac{1}{\alpha_1}r^2\diff\theta + \tfrac{1}{\alpha_1}\big(1+2\pi\tfrac{\alpha'_1}{\alpha_1}r^2\big) \diff s\]
This can be obtained by means of a Moser's trick, see \cite[Lemma~4.5]{Cristofaro-Gardiner:2020aa}.
Notice that every orbit of the Reeb flow $\phi_{\lambda_0}^t$ on $Y_1$ has minimal period $1$, except possibly $\gamma_1=\{0\}\times S^1$ that has minimal period $1/\alpha_1$ (the case in which $\gamma_1$ is regular is allowed, and corresponds to the Seifert invariants $(\alpha_1,\beta_1)=(1,0)$).

For any pair of coprime integers $p_0\neq0$ and $q_0$ such that
\begin{align}
\label{e:transversality}
 \frac {q_0}{p_0}< -\frac{\alpha_1'}{\alpha_1},
\end{align}
and for any $s_0\in S^1$, we introduce the map
\begin{align*}
\iota:[0,\rho)\times S^1\to B_1\times S^1,\qquad \iota(r,s)=\big( r e^{2\pi (s_0+q_0 s)i},p_0s \big),
\end{align*}
which satisfies the following properties:
\begin{itemize}
\item \textit{$($Transversality$\,)$} The restriction $\iota|_{(0,\rho)\times S^1}$ is an embedding transverse to the Reeb vector field $R_{\lambda_0}$, and the image $\iota((0,\rho)\times S^1)$ intersects  $-\alpha_1q_0-p_0\alpha_1'>0$ times every Reeb orbit in $Y_1$ other than $\gamma_1$. Therefore, the 2-form $\iota^*\diff\lambda_0$ is nowhere vanishing on the interior $(0,\rho)\times S^1$, and we employ it to orient the annulus $[0,\rho)\times S^1$.\vspace{2pt}

\item \textit{$($Boundary$\,)$} The  restriction $\iota|_{\{0\}\times S^1}$ is an orientation preserving $p_0$-th fold covering map of the closed Reeb orbit $\gamma_1$. Here, the boundary circle $\{0\}\times S^1$ is oriented by means of the 1-form $\iota^*\lambda_0$.
\end{itemize}
We call $\iota:[0,\rho)\times S^1\to B^2\times S^1$ a $(p_0,q_0)$-local surface of section with boundary on the orbit $\gamma_1$.

We now provide the construction of a suitable global surface of section for the Besse contact 3-manifold $(Y,\lambda)$ with boundary on the orbit $\gamma_1$. An alternative construction in the special case of a regular orbit in $S^3$ was provided by Albach-Geiges \cite[Example 5.6]{Albach:2021wi}.  The next result is a more precise version of Theorem \ref{mt:surface} from the Introduction.

\begin{thm}
\label{t:sos_Besse}
Let $(Y,\lambda_0)$ be a Besse contact 3-manifold with minimal common Reeb period $1$, and $\gamma_1$ be any of its closed Reeb orbits. We denote by $\alpha_1>0$ the integer whose reciprocal $1/\alpha_1$ is the minimal period of $\gamma_1$.
Then, there exist integers $b>0$, $p_0>0$, and $q_0$ with $\gcd(p_0,q_0)=1$, a compact connected oriented surface $\Sigma$ with $b$ boundary components and a global surface of section $\iota:\Sigma\to Y$ for $(Y,\lambda_0)$ satisfying the following properties: 
\begin{itemize}
 
\item[$(i)$] Any connected component $C$ of the boundary $\partial\Sigma$ has a collar neighborhood $A_C\cong[0,\rho)\times S^1$ such that the restriction $\iota|_{A_C}$ is a $(p_0,q_0)$-local surface of section with boundary on the orbit $\gamma_1$. In particular, all boundary components are positively oriented. \vspace{2pt}
 
\item[$(ii)$] Any regular closed Reeb orbit in $Y\setminus\gamma_1$ intersects the image $\iota(\mathrm{int}(\Sigma))$ in $\alpha$ points, where
\begin{align*}
\alpha:=- \frac{b\, p_0}{\eul(Y)\, \alpha_1}>0.
\end{align*}
Here, $\eul(Y)$ is the Euler number of $(Y,\lambda_0)$.\vspace{2pt}

\item[$(iii)$] The restriction $\iota|_{\partial\Sigma}$ is a covering map of $\gamma_1$ of degree $b\,p_0$.

\end{itemize}
\end{thm}

\begin{figure}
\begin{center}
\begin{tiny}
\includegraphics{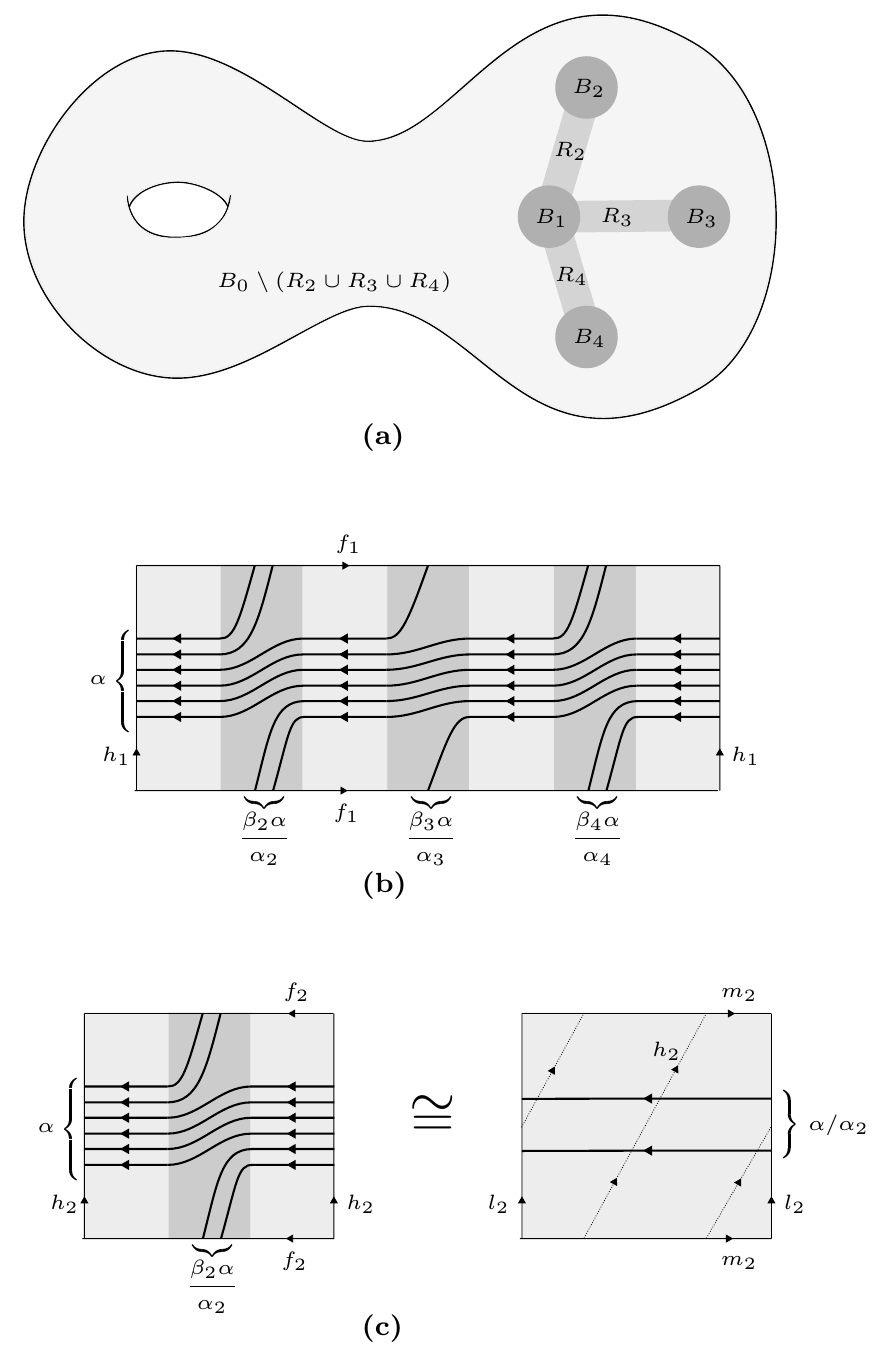}
\end{tiny}
\begin{small}
\caption{Construction of a global surface of section.}
\label{f:sos}
\end{small}
\end{center}
\end{figure}

\begin{proof}
Let $(\alpha_1,\beta_1),....,(\alpha_k,\beta_k)$ be the Seifert invariants of $(Y,\lambda_0)$. Here, we assume without loss of generality that $\gamma_1$ is the Reeb orbit corresponding to the Seifert pair $(\alpha_1,\beta_1)$ (as we already pointed out, $\gamma_1$ is allowed to be a regular orbit, and in that case we have $(\alpha_1,\beta_1)=(1,0)$). For every Seifert pair $(\alpha_j,\beta_j)$, we denote by $(\alpha_j',\beta_j')$ the dual pair satisfying~\eqref{e:dual_Seifert_pair}. We set 
\begin{align}
\nonumber
\alpha & :=\lcm(\alpha_2,...,\alpha_k)>0,\\
\nonumber
\beta & :=\big(\tfrac{\beta_2}{\alpha_2}+...+\tfrac{\beta_k}{\alpha_k}\big)\alpha,\\
\label{e:p}
p & :=\beta_1\alpha + \alpha_1\beta = -\eul(Y)\,\alpha_1\alpha>0,\\
\label{e:q}
q & := -\beta_1'\alpha - \alpha_1'\beta,\\
\nonumber
b & := \gcd(p,q),\\
\nonumber
p_0& :=p/b,\\
\nonumber
q_0& :=q/b.
\end{align}
By inverting the linear Equations~\eqref{e:p} and~\eqref{e:q}, we have 
\begin{align*}
- \alpha & = \alpha_1q +\alpha_1' p, \\
 \beta & = \beta_1 q + \beta_1' p. 
\end{align*}
Notice that~\eqref{e:transversality} is satisfied, for
\begin{align*}
\frac{q_0}{p_0}+\frac{\alpha_1'}{\alpha_1}
=
\frac{q}{p}+\frac{\alpha_1'}{\alpha_1}
=
\frac{\alpha_1 q + \alpha_1' p}{\alpha_1 p}
=
-\frac{\alpha}{\alpha_1 p}<0
\end{align*}
Moreover, 
\begin{align*}
\alpha=- \frac{b\, p_0}{\eul(Y)\, \alpha_1}.
\end{align*}

With the notation of Section~\ref{s:surgery}, we consider the small disks $B_i\subset B$ containing $\pi\circ\gamma_i$ in their interior. We connect $B_1$ with every other $B_i$ by means of a rectangle $R_i$ as shown in Figure~\ref{f:sos}(a). We first define the intersection of our desired surface of section with the solid torus $Y_1=\pi^{-1}(B_1)$ as $b$-many $(p_0,q_0)$-local surfaces of section with boundary on $\gamma_1$; of course, away from their boundary on $\gamma_1$, such local surfaces of section are disjoint. We shall extend these $b$-many components outside $Y_1$ in such a way to create a (connected) global surface of section.

On the torus $\pi^{-1}(\partial_1 B_0)=\pi^{-1}(\partial B_1)$, our defined surface of section winds $- \alpha$-times around $f_1$ and $ \beta$-times around $h_1$. We distribute the windings around $h_1$ into $k-1$ groups as in the example of Figure~\ref{f:sos}(b), where the shaded regions are the annuli $\pi^{-1}(\partial_1 B_0 \cap R_i)$, $i=2,...,k$. We set $B_0':=B_0\setminus (R_2\cup...\cup R_k)$, and extend the surface of section to $\pi^{-1}(B_0')\cong B_0'\times S^1$ as $\alpha$ distinct constant sections 
\[B_0'\times\{t_1\},..., B_0'\times\{t_{\alpha}\};\] 
On $\pi^{-1}(R_i)$, we extend the section in $\pi^{-1}(\partial_1 B_0 \cap R_i)$ by taking the product with an interval.  As a result, in each torus $\pi^{-1}(\partial B_i)$, $i=2,...,k$, our surface of section winds $\alpha$  times around $f_i$, and $\alpha\beta_i/\alpha_i$  times around $h_i$ (Figure~\ref{f:sos}(c)); this means that the surface of section winds $\alpha/\alpha_i$ times around $m_i$, and zero times around $l_i$. Therefore, we can extend it to $\pi^{-1}(B_i)\cong B_i\times S^1$ as $\alpha/\alpha_i$ many distinct constant sections 
\[B_i\times\{t_{i,1}\},...,B_i\times\{t_{i,\alpha/\alpha_i}\}.\]

It remains to show that the resulting surface of section is connected. For each $i=2,...,k$, in $\pi^{-1}(\partial B_i)$ as in the left of Figure~\ref{f:sos}(c) outside the shaded region, the $j$-th horizontal line is connected with the $(j+ \beta_i\alpha/\alpha_i)$-th one modulo $\alpha$. Hence, in order the prove the claim, it suffices to remark that 
\[
\gcd\big( \alpha, \tfrac{\beta_2 \alpha}{\alpha_2},..., \tfrac{\beta_k \alpha}{\alpha_k}\big) = 1.
\]
This latter equality follows from the fact that $\alpha_i$ and $\beta_i$ are coprime and $ \alpha=\lcm(\alpha_2,...,\alpha_k)$. 
\end{proof}

Now, we consider the global surface of section $\iota:\Sigma\to Y$ given by Theorem~\ref{t:sos_Besse}, and the differential forms 
\begin{align*}
\nu_0:=\iota^*\lambda_0,
\qquad
\omega_0:=\diff\nu_0.
\end{align*}
Since the Reeb vector field $R_{\lambda_0}$ is transverse to $\iota(\mathrm{int}(\Sigma))$, we readily infer that $\omega_0$ is symplectic on $\mathrm{int}(\Sigma)$. A simple computation shows that $\omega_0$ vanishes at all points of $\partial\Sigma$, and indeed satisfies Assumption~\ref{assumption} from Section~\ref{ss:fixed_point_thm}. This is a consequence of the fact that our surface of section restricts to a $(p_0,q_0)$-local surface of section near any connected component $C$ of $\partial\Sigma$ (Theorem~\ref{t:sos_Besse}(i)). Indeed, consider a collar neighborhood $A_C\subset \Sigma$ of $C$, and the solid cylinder neighborhood $Y_1\subset Y$ of $\gamma_1$. Under suitable identifications 
\[A_C\equiv [0,\rho)\times S^1,\qquad Y_1\equiv B_1\times S^1,\]
the restriction $\iota|_{A_C}:A_C\to Y$ has the form
\begin{align*}
\iota|_{A_C}(r,s)=\big( r e^{2\pi q_0 s i},p_0s \big),
\end{align*}
and $\omega_0|_{A_C}$ can be written as
\begin{align}
\label{e:omega_0_annulus}
\omega_0|_{A_C}
 &=
 \underbrace{\tfrac{4\pi}{\alpha_1} \big( q_0+\tfrac{\alpha_1'}{\alpha_1}p_0\big)}_{(*)} r\,\diff r\wedge\diff s.
\end{align}
The constant $(*)$ is strictly negative, and up to rescaling in the $r$ variable, it can be replaced by $-1$, as in Assumption~\ref{assumption}.

The global surface of section $\iota:\Sigma\to Y$ induces a surjective map
\begin{align*}
 \tilde \iota:\Sigma\times S^1\to Y,\qquad \tilde\iota(z,t)=\phi_{\lambda_0}^t(\iota(z)),
\end{align*}
which restricts to an $\alpha$-th fold branched covering map  
\begin{align}
\label{e:alpha_covering_map} 
\tilde \iota|_{\mathrm{int}(\Sigma)\times S^1}:\mathrm{int}(\Sigma)\times S^1\to Y\setminus\gamma_1.
\end{align} 
Here,
\begin{align*}
 \alpha=- \frac{b\, p_0}{\eul(Y)\, \alpha_1}>0
\end{align*}
is the number of intersections of any regular Reeb orbit in $Y\setminus\gamma_1$ with $\iota(\mathrm{int}(\Sigma))$, according to Theorem~\ref{t:sos_Besse}(ii). The branch set has the form $\Sigma_{\mathrm{sing}}\times S^1$, where $\Sigma_{\mathrm{sing}}$ is the finite set of points in $\mathrm{int}(\Sigma)$ which are mapped by $\iota$ to points on singular orbits. The map $\tilde\iota$ is a local diffeomorphism also at the branch points.

We denote by $\pr_1:\Sigma\times S^1\to\Sigma$ the projection onto the first factor.
The one-form 
\begin{align}
\label{e:tilde_lambda_0}
\tilde\lambda_0 
:= 
\tilde\iota^*\lambda_0=\diff t+\pr_1^*\nu_0
\end{align}
is a contact form on the interior $\mathrm{int}(\Sigma)\times S^1$, with associated Reeb vector field 
\[R_{\tilde\lambda_0}=\partial_t.\] 
Notice in particular that $R_{\tilde\lambda_0}$ is well-defined and smooth up to the boundary $\partial \Sigma\times S^1$ as well.

The global surface of section allows to compute the volume of our Besse contact manifold. This was pointed out by Geiges \cite{Geiges:2020aa}, and we provide the details here for the reader's convenience.
\begin{lem} 
\label{l:volume_Besse}
If $(Y,\lambda_0)$ is a Besse contact 3-manifold whose closed Reeb orbits have minimal common period 1, its volume is equal to the negative of its Euler number. More precisely, if $\iota: \Sigma \rightarrow Y$ is a global surface of section as above, we have:
\begin{align*}
 \mathrm{vol}(Y,\lambda_0) = \frac{1}{\alpha} \mathrm{area}(\Sigma,\omega_0) = -\eul(Y).
\end{align*}
\end{lem}

\begin{proof}
The contact volume form $\lambda_0\wedge\diff\lambda_0$ is pulled back to
\begin{align*}
 \tilde\iota^*(\lambda_0\wedge\lambda_0)
 =
 \tilde\lambda_0\wedge\diff\tilde\lambda_0
 =
 \diff t\wedge\pr_1^*\diff\nu_0.
\end{align*}
We recall that $\tilde\iota$ restricts to an $\alpha$-th fold branched covering map \eqref{e:alpha_covering_map}. Therefore
\begin{align*}
\alpha\,\vol(Y,\lambda_0)
= 
\alpha
\int_{Y}\lambda_0\wedge\diff\lambda_0
=
\int_{\Sigma\times S^1} \!\!\!\tilde\iota^*(\lambda_0\wedge\diff\lambda_0)
=
\int_{\Sigma\times S^1} \!\!\! \diff t\wedge \pr_1^*\omega_0
=\int_\Sigma \omega_0,
\end{align*}
and the latter term is precisely $\area(\Sigma,\omega_0)$. Moreover, since $\iota|_{\partial\Sigma}$ is an orientation preserving $b\,p_0$-th fold covering map of the $1/\alpha_1$-periodic Reeb orbit $\gamma_1$, Stokes' theorem implies
\[
\int_\Sigma \omega_0 
=\int_{\partial\Sigma} \nu_0
=
b\,p_0\int_{\gamma_1} \lambda_0
=
\frac{b\,p_0}{\alpha_1}
=
-\alpha\, \eul(Y)
.
\qedhere
\]
\end{proof}

\subsection{From Besse to nearly Besse contact forms}\label{s:nearly_Besse}
Once the existence of a global surface of section for Besse contact 3-manifolds established, the construction of global surfaces of section for nearly-Besse contact 3-manifolds will be a generalization of the one of nearly-Zoll contact 3-manifolds provided in \cite{Abbondandolo:2018fb,Benedetti:2021aa}. We work out the details in this section.

We consider the global surface of section $\iota:\Sigma\to Y$ of the Besse contact manifold $(Y,\lambda_0)$ with boundary on $\gamma_1$, and its related objects from the previous subsection.
Let $\lambda$ be a contact form on $Y$ such that
\begin{align}
\label{e:same_Reeb_on_gamma_1}
 R_{\lambda}|_{\gamma_1}=R_{\lambda_0}|_{\gamma_1}
\end{align}
We set
\[
\tilde{\lambda} := \tilde\iota^* \lambda, \quad \nu := \iota^* \lambda, \quad \omega := d\nu = \iota^* d\lambda.
\]
We recall that the analogous differential forms for the Besse contact form $\lambda_0$ are denoted by $\tilde\lambda_0$, $\nu_0$, and $\omega_0$. Since $\tilde\iota$ is a local diffeomorphism on $\mathrm{int}(\Sigma) \times S^1$, $\tilde{\lambda}$ is a contact form on this open manifold. 

The next lemma is analogous to \cite[Prop.~3.6]{Abbondandolo:2018fb} and \cite[Prop.~3.10]{Benedetti:2021aa}.

\begin{lem}\label{l:nearly_Besse_normalized}
$ $\nopagebreak
\begin{itemize}

\item[$(i)$] The pull-back of $\nu$ via the inclusion $\partial\Sigma\hookrightarrow\Sigma$ satisfies
\[\nu|_{\partial\Sigma}=\nu_0|_{\partial\Sigma}.\]

\item[$(ii)$] The 2-form $\omega$ vanishes at all points of $\partial\Sigma$, i.e.
\begin{align*}
\omega_z= 0,\qquad \forall z\in \partial\Sigma.
\end{align*}

\item[$(iii)$] The Reeb vector field $R_{\tilde\lambda}$, a priori only defined on the interior of $\Sigma\times S^1$, admits a smooth extension to $\Sigma\times S^1$ that is tangent to the boundary $\partial \Sigma\times S^1$.

\item[$(iv)$] For all $\epsilon>0$ there exists $\delta>0$ such that, if the above contact form $\lambda$ further satisfies $\|R_\lambda-R_{\lambda_0}\|_{C^2}<\delta$, then $\|R_{\tilde\lambda}-\partial_t\|_{C^1}<\epsilon$. Here, the norms are associated with arbitrary fixed Riemannian metrics.
\end{itemize}
\end{lem}

\begin{proof}
Let us consider a collar neighborhood $A_C\times S^1$ of a connected component $C\times S^1$ of $\partial\Sigma \times S^1$, and the solid cylinder neighborhood $Y_1\subset Y$ of $\gamma_1$. With the identifications $A_C\equiv[0,\rho)\times S^1$, $Y_1\equiv B_1\times S^1$, and $\gamma_1\equiv\{0\}\times S^1$, as in the previous subsection, the restriction of the map $\tilde\iota$ can be written as
\begin{align*}
\tilde\iota|_{A_C\times S^1}:A_C\times S^1\to Y_1\equiv B_1\times S^1,
\qquad
\tilde\iota(r,s,t)=\big( r e^{2\pi(q_0s-\alpha_1't)i} , p_0s+\alpha_1t \big).
\end{align*}
and $\tilde\lambda_0|_{A_C\times S^1}$ can be written as
\begin{align}
\label{e:pull_back_lambda_0}
\tilde\lambda_0|_{A_C\times S^1}
:= 
\diff t
+
\Big(\tfrac{p_0}{\alpha_1} +\tfrac{2\pi}{\alpha_1} \big( q_0+\tfrac{\alpha_1'}{\alpha_1}p_0 \big) r^2\Big)\diff s.
\end{align}
The computations 
\begin{align*}
 \tilde\iota_*\partial_t|_{C\times S^1} =R_{\lambda_0}=R_\lambda,
 \qquad
 \tilde\iota_*\partial_s|_{C\times S^1} = p_0\partial_s = \tfrac{p_0}{\alpha_1} R_{\lambda_0}= \tfrac{p_0}{\alpha_1} R_\lambda,
\end{align*}
imply 
\begin{align*}
 \tilde\lambda(\partial_t)|_{C\times S^1}=\tilde\lambda_0(\partial_t)|_{C\times S^1} \equiv 1,
 \qquad
 \tilde\lambda(\partial_s)|_{C\times S^1} = \tilde\lambda_0(\partial_s)|_{C\times S^1} \equiv \tfrac{p_0}{\alpha_1},
\end{align*}
and
\begin{align*}
(\partial_t\,\lrcorner\,\diff\tilde\lambda)_z = \tilde\iota^*(R_{\lambda_0}\,\lrcorner\,\diff\lambda)_z &= \tilde\iota^*(R_\lambda\,\lrcorner\,\diff\lambda)_z \equiv 0,\\
(\partial_s\,\lrcorner\,\diff\tilde\lambda)_z 
= 
\tfrac{p_0}{\alpha_1}\tilde\iota^*(R_{\lambda_0}\,\lrcorner\,\diff\lambda)_z 
&= 
\tfrac{p_0}{\alpha_1}\tilde\iota^*(R_\lambda\,\lrcorner\,\diff\lambda)_z \equiv 0,\\
&\qquad\qquad\forall z\in C\times S^1.
\end{align*}
These identities, together with $\nu_0=\tilde\lambda_0|_{\Sigma\times\{0\}}$ and $\nu=\tilde\lambda|_{\Sigma\times\{0\}}$, readily imply points (i) and (ii).

As for point (iii), let us write $R_\lambda|_{Y_1}$ in coordinates $(x+iy,s)=(re^{i\theta},s)$ as
\begin{align*}
R_\lambda|_{Y_1}
&=  F_1\partial_x + F_2\partial_y + \partial_s\\
&= (F_1\cos(\theta)+F_2\sin(\theta))\,\partial_r + \tfrac1r(F_2\cos(\theta) - F_1\sin(\theta)) \partial_\theta + F_3\,\partial_s,
\end{align*}
for some smooth functions $F_j:Y_1\to\R$. Since $R_\lambda$ and $R_{\lambda_0}$ coincide along the closed Reeb orbit $\gamma_1$, we have
\begin{align*}
 F_1(0)=F_2(0)=0,\qquad F_3(0)=\alpha_1.
\end{align*}
Since
\begin{align*}
\tilde\iota_*\partial_r=\partial_r,
\qquad
\tilde\iota_*\partial_s=2\pi q_0\,\partial_\theta + p_0\,\partial_s,
\qquad
\iota_*\partial_t= -2\pi\alpha_1'\,\partial_\theta + \alpha_1\,\partial_s,
\end{align*}
the Reeb vector fiels $R_{\tilde\lambda}$ on the interior $\mathrm{int}(A_C)\times S^1$ is given by
\begin{align*}
 R_{\tilde\lambda}=f_1\,\partial_r + f_2\,\partial_s + f_3\,\partial_t,
\end{align*}
where, if we set 
\begin{align*}
 \theta(s,t) :=q_0s-\alpha_1't,\qquad G_j(r,s,t):=\frac{F_j\circ\iota(r,s,t)}{r},\ \  j=1,2,
\end{align*}
the functions $f_j:\mathrm{int}(A_C)\times S^1\to\R$ are given by
\begin{align}
\label{e:f1}
 f_1 & = \cos(\theta)F_1\circ\tilde\iota+\sin(\theta)F_2\circ\tilde\iota,\\
\label{e:f2} 
f_2 & = \big( 2\pi\big(q_0+\tfrac{\alpha_1'}{\alpha_1}p_0\big)\big)^{-1} \left( -\sin(\theta)G_1 +\cos(\theta)G_2   + 2\pi\tfrac{\alpha_1'}{\alpha_1} F_3\circ\tilde\iota\right),\\
\label{e:f3}
 f_3& = \tfrac1{\alpha_1}\big(F_3\circ\tilde\iota-p_0\,f_2\big).
\end{align}
Since $F_1\circ\tilde\iota|_{C\times S^1}=F_2\circ\tilde\iota|_{C\times S^1}\equiv0$, the functions $G_1$ and $G_2$ extend smoothly to the whole $A_C\times S^1$, and so do the functions $f_1,f_2,f_3$. Moreover, $f_1|_{C\times S^1}\equiv0$. This proves that $R_{\tilde\lambda}$ extends smoothly to a vector field on $\Sigma\times S^1$ that is tangent to the boundary $\partial\Sigma\times S^1$.

Finally, assume that $R_\lambda$ is $C^2$-close to $R_{\lambda_0}$. Away from any fixed neighborhood of $\partial\Sigma\times S^1$, $R_{\tilde\lambda}$ is $C^2$-close to $R_{\tilde\lambda_0}=\partial_t$. On $Y_1$, since $R_\lambda=-2\pi\alpha_1'\partial_\theta+\alpha_1\partial_s$, the functions
\begin{align*}
F_1-2\pi\alpha_1' r\sin(\theta),\qquad
F_2+2\pi\alpha_1' r\cos(\theta),\qquad
F_3-\alpha_1
\end{align*}
are $C^2$-small and vanish at $\partial\Sigma\times S^1$. Therefore, the function \[-\sin(\theta)G_1+\cos(\theta)G_2+2\pi\tfrac{\alpha_1'}{\alpha_1} F_3\circ\tilde\iota\] is $C^1$-small. Equation \eqref{e:f1} implies that $f_1$ is $C^2$-small, and Equations \eqref{e:f2} and \eqref{e:f3} imply that $f_2$ and $f_3-1$ are $C^1$-small. Overall, we conclude that $R_{\tilde\lambda}$ is $C^1$-close to $R_{\tilde\lambda_0}=\partial_t$.
\end{proof}

Besides~\eqref{e:same_Reeb_on_gamma_1}, we now assume:
\begin{ass}
\label{a:transversality}
The vector field $R_{\tilde\lambda}$ on $\Sigma\times S^1$ is transverse to $\Sigma\times \{s\}$ and oriented as $\partial_s$, for every $s\in S^1$.\hfill\qed
\end{ass}

Thanks to Lemma \ref{l:nearly_Besse_normalized}(iv), Assumption~\ref{a:transversality} is implied by the $C^2$-closeness of $R_{\lambda}$ to $R_{\lambda_0}$. By Assumption~\ref{a:transversality}, the first-return time
\begin{align}
\label{e:first_return_time}
 \tau:\Sigma\to(0,\infty),
 \qquad
 \tau(z):=\min\big\{t>0\ \big|\ \phi_{\tilde\lambda}^t(z,0)\in\Sigma\times\{0\}\big\},
\end{align}
is a well-defined smooth map, and the first-return map
\begin{align}
\label{e:first_return_map}
 \phi:\Sigma\to \Sigma,
 \qquad
 (\phi(z),0)=\phi_{\tilde\lambda}^{\tau(z)}(z,0),
\end{align}
is a diffeomorphism. The diffeomorphism $\phi$ is isotopic to the identity through the isotopy $\{\phi_s\}$ defined by $\phi_0:= \mathrm{id}$ and, for $s\in (0,1]$,
\begin{align}
\label{isotopy}
(\phi_s(z),s)=\phi_{\tilde\lambda}^{\tau_s(z)}(z,0),
\end{align}
where
\[
\tau_s(z):=\min\big\{t>0\ \big|\ \phi_{\tilde\lambda}^t(z,0)\in\Sigma\times\{s\}\big\}.
\]
In the particular case $\lambda=\lambda_0$, we would get that $\tau$ is identically equal to 1 and $\phi_s$ equals the identity on $\Sigma$ for every $s\in [0,1]$. 

The next lemma relates the volume of $(Y,\lambda)$ to the integral of the first return time.

\begin{lem}
\label{l:volume}
The contact volume of $(Y,\lambda)$ is given by
\begin{align*}
\vol(Y,\lambda) = \frac1\alpha \int_\Sigma \tau\,\omega.
\end{align*}
\end{lem}

\begin{proof} 
The bijective map
\[
j: \Sigma \times [0,1) \rightarrow \Sigma \times S^1, \qquad j(z,s) := \phi_{\tilde{\lambda}}^{s\tau(z)} (z,0),
\]
satisfies
\[
j^*(\tilde\lambda \wedge d\tilde\lambda) = \tau\, ds \wedge \mathrm{pr}_1^* \omega,
\]
where $\pr_1:\Sigma\times [0,1) \to\Sigma$ is the projection onto the first factor. Together with the fact that the restriction of $\tilde{\iota}$ to the interior of $\Sigma\times S^1$
is an $\alpha$-th fold branched covering map of a full measure subset of $Y$ pulling the volume form $\lambda\wedge d\lambda$ back to $\tilde\lambda \wedge d\tilde\lambda$, we obtain
\begin{align*}
\alpha\, \vol(Y,\lambda) &= \alpha \int_Y \lambda \wedge d\lambda = \int_{\Sigma\times S^1} \tilde\lambda \wedge d\tilde\lambda = \int_{\Sigma\times [0,1)} j^*( \tilde\lambda \wedge d\tilde\lambda) \\ &= \int_{\Sigma\times [0,1)} \tau\, ds \wedge \mathrm{pr}_1^* \omega = \int_{\Sigma} \tau \, \omega.
\end{align*}
\end{proof}

The next lemma relates the first return map $\phi$ to the first return time $\tau$ via the 1-form $\nu$.

\begin{lem} 
\label{l:return_time_map}
The first-return map $\phi$ is an exact symplectomorphism of $(\Sigma,\omega)$, and more precisely
\begin{align*}
 \phi^*\nu =\nu+\diff\tau.
\end{align*}
The boundary restriction of the first return time $\tau$ is given by
\begin{align*}
 \tau(z)=1+\int_{\{s\mapsto \phi_s(z)\}} \!\!\! \nu,\qquad\forall z\in\partial\Sigma.
\end{align*}
\end{lem}

\begin{proof}
The first statement follows by the well-known computation
\begin{align*}
 \phi^*\nu = (\phi_{\tilde\lambda}^{\tau(z)})^*\tilde\lambda|_{\Sigma\times\{0\}} + \phi^*(\tilde\lambda(R_{\tilde\lambda}))\diff\tau=\nu+\diff\tau.
\end{align*}
For each $z\in\partial\Sigma$, the curve
\[
\zeta_z:[0,1]\to \partial\Sigma\times S^1,\qquad \zeta_z(s):=(\phi_s(z),s) = \phi_{\tilde{\lambda}}^{\tau_s(z)} (z,0),
\] 
is a reparametrization of the restriction of the orbit of $(z,0)$ by the Reeb flow of $\tilde{\lambda}$ to the interval $[0,\tau(z)]$, which makes one full turn around the second factor of $\partial\Sigma\times S^1$. Therefore,
\begin{align*}
\tau(z)
=\int_{\zeta_z} \tilde\lambda
=\int_{\zeta_z} \tilde\lambda_0
=\int_{\zeta_z} \!\!\big( \diff s + \pr_1^*\nu \big)
=\int_{\zeta_z} \diff s + \int_{\pr_1\circ\zeta_z} \!\!\!\nu
=1 + \int_{\{s\mapsto \phi_s(z)\}}\!\!\! \nu,
\end{align*}
where the second equality follows by Lemma~\ref{l:nearly_Besse_normalized}(i), and the third one from~\eqref{e:tilde_lambda_0}.
\end{proof}

In order to prove Theorem~\ref{t:main}, we will need to apply the fixed point Theorem~\ref{t:fixed_point}, which concerns symplectomorphisms that are $C^1$-close to the identity on a surface $\Sigma$ equipped with a fixed 2-form symplectic in the interior and vanishing in a suitable way at the boundary. By Lemma \ref{l:return_time_map}, the diffeomorphism $\phi: \Sigma \rightarrow \Sigma$ is symplectic with respect to the 2-form $\omega=\iota^* d\lambda$, which varies with $\lambda$. However, assumption~\eqref{e:same_Reeb_on_gamma_1} and its consequence $\nu|_{\partial \Sigma} = \nu_0|_{\partial \Sigma}$ from Lemma~\ref{l:nearly_Besse_normalized}(i) imply that
\[
\mathrm{area}(\Sigma,\omega) = \int_{\partial \Sigma} \nu = \int_{\partial \Sigma} \nu_0 = \mathrm{area}(\Sigma,\omega_0).
\]
Therefore, we can conjugate $\phi$ by a diffeomorphism $\kappa: \Sigma\rightarrow \Sigma$ pulling $\omega$ back to $\omega_0$ and obtain a symplectomorphism with respect to the fixed 2-form $\omega_0$ on $\Sigma$. The construction of this diffeomorphism and the proof of its further properties are based as usual on Moser's trick but require a bit of care, since we are working on a surface with boundary. We work out the details in the following lemma, which is a variation of \cite[Prop.~3.9]{Benedetti:2021aa}.

\begin{lem} 
\label{l:pullback_to_nu0}
If $\lambda$ is $C^2$-close enough to $\lambda_0$, then there exists a diffeomorphism $\kappa:\Sigma\to\Sigma$ such that $\kappa|_{\partial\Sigma}=\id$ and $\kappa^*\omega=\omega_0$. Moreover, $\kappa$ $C^1$-converges to the identity as $\lambda$ $C^2$-converges to $\lambda_0$.
\end{lem}

\begin{proof}
Note that the smallness of $\|\lambda-\lambda_0\|_{C^2}$ implies the smallness of $\|\nu-\nu_0\|_{C^2}$ and $\|\omega-\omega_0\|_{C^1}$. Assumption~\ref{a:transversality} guarantees that $\omega$ is a symplectic form in the interior of $\Sigma$ inducing the same orientation as $\omega_0$. Therefore, the 2-forms $\omega_t:=t\omega+(1-t)\omega_0$ are symplectic on $\mathrm{int}(\Sigma)$ for every $t\in[0,1]$. They are actually uniformly $C^1$-close to $\omega_0$ when $\|\lambda-\lambda_0\|_{C^2}$ is small. 

We look for an isotopy $\kappa_t:\Sigma\to\Sigma$ such that $\kappa_t|_{\partial\Sigma}\equiv\mathrm{id}$ and $\kappa_t^*\omega_t=\omega_0$. We build the time-dependent vector field $X_t$ realizing such isotopy, i.e.\ $\tfrac{\diff}{\diff t}\kappa_t=X_t\circ\kappa_t$. By differentiating $\kappa_t^*\omega_t$ with respect to $t$, we obtain
\begin{align}
\label{e:X_Moser_0}
0=\tfrac{\diff}{\diff t}\kappa_t^*\omega_t
=
\kappa_t^* \big(  \diff(X\,\lrcorner\,\omega_t) + \omega-\omega_0  \big) 
=
\kappa_t^*  \diff\big(X\,\lrcorner\,\omega_t + \nu-\nu_0 \big)  .
\end{align}
We define $X_t$ on $\mathrm{int}(\Sigma)$ by the equation
\begin{align}
\label{e:X_Moser}
X_t\,\lrcorner\,\omega_t = \nu_0 - \nu + \diff f_t
\end{align}
for a suitable $C^2$-small smooth function $f:\Sigma\times [0,1] \to\R$ to be determined. A suitable choice of $f$ will guarantee that $X_t$ has a smooth extension to the whole $\Sigma$ with $X_t|_{\partial\Sigma}\equiv0$.

For every connected component $C$ of $\partial\Sigma$, we fix a  collar neighborhood $A_C\subset\Sigma$ so that, with the usual suitable identification $[0,\rho)\times S^1$, the differential forms $\omega_0$ can be written as in~\eqref{e:omega_0_annulus}. Actually, up to rescaling the interval $[0,\rho)$, we can even write $\omega_0|_{A_C}$ as 
\begin{align*}
 \omega_0|_{A_C}=-r\,\diff r\wedge\diff s,
\end{align*}
where $r\in[0,\rho)$ and $s\in S^1$. Lemma~\ref{l:nearly_Besse_normalized}(i-ii) implies $\nu|_C=\nu_0|_C$ and $\omega_z=(\omega_0)_z$ for all $z\in C$. Therefore we can write
\begin{align*}
(\nu_0-\nu)|_{A_C} = h_1\,\diff r + r\,h_2\,\diff s,
\qquad
\omega|_{A_C} = -r\, h_3\, \diff r\wedge\diff s,
\end{align*}
for some smooth functions $h_i:A_C\to\R$. If $\nu_0$ and $\nu$ are $C^2$-close, the function $h_1$ is $C^2$-small, while $h_2$ is $C^1$-small. Moreover, since $\omega_0$ and $\omega$ are $C^1$-close, the function $1-h_3(0,\cdot)$ is $C^0$-small. In particular, up to choosing the annulus $A_C$ small enough, $h_3$ is strictly positive on the whole $A_C$. Since $\diff(\nu_0-\nu)=\omega_0-\omega$, we have
\begin{align}
\label{e:d_nu_local_coord}
 \partial_r(r\, h_2)-  \partial_s h_1 = r (h_3 - 1).
\end{align}
The two-form $\omega_t|_{A_C}$ is given by
\begin{align*}
\omega_t|_{A_C}=-r\,(t(h_3-1)+1)\,\diff r\wedge\diff s,
\end{align*}
and if we write the vector field $X_t|_{A_C}$ in $(r,s)$ coordinates as $X_t=R_t\,\partial_r + S_t\,\partial_s$, Equation~\eqref{e:X_Moser} becomes
\begin{align*}
 R_t  = -\frac {r\, h_2 + \partial_s f}{r\,(t(h_3-1)+1)},\qquad
 S_t  = \frac { h_1 + \partial_r f}{r\,(t(h_3-1)+1)}.
\end{align*}
We now choose $f:\Sigma\to\R$ to be a smooth function such that $f|_{\partial\Sigma}\equiv0$ and, on any collar neighborhood $A_C=[0,\rho)\times S^1$ as above, satisfies
\begin{align*}
f(r,s)
=
\left\{
  \begin{array}{@{}ll}
    \displaystyle -\int_0^r h_1(x,s)\,\diff x,    & \mbox{if }r\leq\tfrac13\rho, \vspace{5pt} \\ 
    0, & \mbox{if }r\geq\tfrac23\rho. 
  \end{array}
\right.
\end{align*}
We shall choose such an $f$ so that $\|f\|_{C^2}\leq\const\|h_1\|_{C^2}$, and in particular $f$ is $C^2$-small since $\nu$ and $\nu_0$ are $C^2$-close. With this choice of $f$, we have $S_t(r,s)=0$ if $r\leq\tfrac13\rho$. As for the function $R_t$, for all $r\leq\tfrac13\rho$ Equation~\eqref{e:d_nu_local_coord} implies
\begin{align*}
 R_t(r,s)
 &=
 -\frac {r\, h_2(r,s) + \partial_s f(r,s)}{r\,(t(h_3-1)+1)}
 =
 -\frac {r\, h_2(r,s)  -\int_0^r \partial_s h_1(x,s)\,\diff x}{r\,(t(h_3-1)+1)}\\
 &=
 -\frac {r\, h_2(r,s)  +\int_0^r \big(x (h_3(x,s) - 1) - \partial_x(x\, h_2(x,s))\big)\,\diff x}{r\,(t(h_3-1)+1)}\\
 &=
 -\frac {\tfrac1r\int_0^r x (h_3(x,s) - 1) \,\diff x}{t(h_3-1)+1}.
\end{align*}
We already know that the function $t(h_3-1)+1$ appearing in the denominator of the above equation is nowhere vanishing. As for the numerator, we can rewrite the integral as
\begin{align*}
\int_0^r x (h_3(x,s) - 1) \,\diff x = r^2\, h_4(r,s)
\end{align*}
for some $C^0$-small smooth function $h_4:A_C\to\R$ such that $h_4(0,s)=h_3(0,s) - 1$. Therefore
\begin{align*}
 R_t(r,s) = -\frac {r\,h_4(r,s)}{t\,(h_3-1)+1}.
\end{align*}
From this expression we readily infer that $R_t$ is $C^1$-small, extends smoothly to the whole $\Sigma$, and $R_t|_{\partial\Sigma}\equiv0$. Summing up, we obtained a $C^1$-small smooth vector field $X_t$ on $\Sigma$ satisfying~\eqref{e:X_Moser} and $X|_{\partial\Sigma}\equiv0$. 
Its flow $\kappa_t$ is $C^1$-small for all $t\in[0,1]$, and satisfies $\kappa_t|_{\partial\Sigma}=\id$ and, by~\eqref{e:X_Moser_0}, $\kappa_t^*\omega_t=\omega_0$.
\end{proof}

The following proposition sums up the arguments of this section and will play a crucial role in the proof of Theorem~\ref{t:main}(ii). 

\begin{prop}
\label{p:sum_up}
Let $\lambda_0$ be a Besse contact form on the closed manifold $Y$ whose closed Reeb orbits have minimal common period $1$, and let $\gamma_1$ be any orbit of $R_{\lambda_0}$. Then there exists a closed surface with boundary $\Sigma$ endowed with an exact 2-form $\omega_0$ which is symplectic on the interior of $\Sigma$ and satisfies Assumption~\ref{assumption} such that the following holds. For every $\epsilon>0$ small enough and for every $C^1$-neighborhood $\UU$ of the identity in $\widetilde{\Ham}_0(\Sigma,\omega_0)$ there exist $\delta>0$ and, for each contact form $\lambda$ on $Y$ such that
\[
R_\lambda|_{\gamma_1}=R_{\lambda_0}|_{\gamma_1}, \qquad \|\lambda-\lambda_0\|_{C^2} < \delta, \qquad \|R_{\lambda} - R_{\lambda_0}\|_{C^2} < \delta,
\]
a global surface of section
\[
j:\Sigma\to Y
\] 
for $R_{\lambda}$ mapping each component of $\partial \Sigma$ onto some positive iterate of $\gamma_1$ and an element 
\[
\tilde\psi\in\UU
\]
with the following properties: 
\begin{itemize}

\item[$(i)$] The normalized Calabi invariant of $\tilde\psi$ is related to the volumes of $(Y,\lambda)$ and $(Y,\lambda_0)$ by
\begin{align*}
\widehat{\Cal}(\widetilde\psi) =   \frac{\vol(Y,\lambda)}{\vol(Y,\lambda_0)} - 1.
\end{align*}

\item[$(ii)$] A point $z\in \mathrm{int}(\Sigma)$ is a contractible fixed point of $\tilde\psi$ if and only if \[\gamma_z(t):=\phi_\lambda^t(j(z))\] is a closed Reeb orbit of $R_{\lambda}$ in $Y\setminus\gamma_1$ with (not necessarily minimal) period 
\[
1+a_{\tilde\psi}(z)\in(1-\epsilon,1+\epsilon).
\]
Here, $a_{\tilde\psi}(z)$ is the normalized action of the contractible fixed point $z$.

\item[$(iii)$] The element $\tilde\psi$ is the identity in $\widetilde{\Ham}_0(\Sigma,\omega_0)$ if and only if $(Y,\lambda)$ is Besse and its Reeb orbits have common period 1.
\end{itemize}
\end{prop}

\begin{proof}
We consider a global surface of section
\[
\iota: \Sigma \rightarrow Y
\]
for the Reeb flow of $\lambda_0$ as in Theorem \ref{t:sos_Besse} and the corresponding map
\[
\tilde\iota: \Sigma \times S^1 \rightarrow Y, \qquad \tilde\iota(z,t) := \phi_{\lambda_0}^t(\iota(z)).
\]
The 2-form
\[
\omega_0 := \iota^* d\lambda_0,
\]
is symplectic in the interior of $\Sigma$, satisfies Assumption \ref{assumption} (see the discussion after the proof of Theorem \ref{t:sos_Besse}) and, thanks to Lemma \ref{l:volume_Besse}, has total area
\begin{equation}
\label{areavol}
\mathrm{area}(\Sigma,\omega_0) = \alpha \, \mathrm{vol}(Y,\lambda_0),
\end{equation}
where $\alpha$ is the positive integer appearing in Theorem \ref{t:sos_Besse}. Given another contact form $\lambda$ on $Y$, we set as before
\[
\tilde{\lambda}:= \tilde\iota^*\lambda, \qquad \nu:= \iota^* \lambda, \qquad \omega:= d\nu = \iota^* d\lambda.
\]
Here, we are assuming that $R_\lambda|_{\gamma_1}=R_{\lambda_0}|_{\gamma_1}$, which is exactly condition~\eqref{e:same_Reeb_on_gamma_1}, and that $\|R_{\lambda} - R_{\lambda_0}\|_{C^2}$ is small enough, so that also Assumption~\ref{a:transversality} holds thanks to Lemma \ref{l:nearly_Besse_normalized}(iv). In particular, $\iota$ is also a global surface of section for the Reeb flow of $\lambda$. Moreover, by Lemma \ref{l:nearly_Besse_normalized} the 1-form $\tilde{\lambda}$ defines a flow on $\Sigma\times S^1$ having $\Sigma\times \{0\}$ as global surface of section and we denote by $\tau$ and $\phi$ the corresponding first return time and first return map, see (\ref{e:first_return_time}) and (\ref{e:first_return_map}). By Lemma \ref{l:nearly_Besse_normalized}(iv), the map $\phi$ is $C^1$-close to the identity when $\|R_{\lambda}-R_{\lambda_0}\|_{C^2}$ is small.

By further assuming that $\|\lambda-\lambda_0\|_{C^2}$ is small enough, we can use Lemma \ref{l:pullback_to_nu0} to find a diffeomorphism $\kappa:\Sigma \rightarrow \Sigma$ that is  $C^1$-close to the identity and satisfies $\kappa^* \omega = \omega_0$. Up to conjugating $\phi$ by $\kappa$ and replacing $\iota$ by $j:= \iota\circ \kappa$, which is still a global surface of section for the Reeb flow of $\lambda$, we may assume that $\omega$ equals $\omega_0$.

In this case, $\nu$ is a primitive of $\omega_0$ and the equality
\begin{equation}
\label{prima}
\phi^* \nu - \nu = d\tau
\end{equation}
proved in Lemma \ref{l:return_time_map} shows that $\phi$ is an exact symplectomorphism on $(\Sigma,\omega_0)$. Being $C^1$-close to the identity, $\phi$ is the image under the universal cover
\[
\pi: \widetilde{\mathrm{Ham}}(\Sigma,\omega_0) \rightarrow \mathrm{Ham}(\Sigma,\omega_0)
\]
of a unique $\tilde{\psi}=[\{\psi_t\}]$ which is also $C^1$-close to the identity (see Theorem \ref{t:quasi_autonomous}). Moreover, the $C^1$-closeness to the identity implies that the Hamiltonian isotopy $\{\psi_t\}$ is homotopic with fixed ends to the (non necessarily symplectic) isotopy $\{\phi_t\}$ which is defined in (\ref{isotopy}), and hence Lemma \ref{l:return_time_map} gives us the identity
\begin{equation}
\label{seconda}
\tau(z) = 1 + \int_{\{t\mapsto \psi_t(z)\}} \nu,\qquad\forall z\in\partial\Sigma.
\end{equation}
Identities (\ref{prima}) and (\ref{seconda}) imply that $\tilde{\psi}$ has vanishing flux (see Remark \ref{whenflux0}), so we may assume that it belongs to the $C^1$-neighborhood $\mathcal{U}$ of the identity in $\widetilde{\mathrm{Ham}}_0(\Sigma,\omega_0)$, and give us the following relationship between the function $\tau$ and the normalized action of $\tilde{\psi}$ with respect to the primitive $\nu$ of $\omega_0$:
\begin{equation}
\label{tauaction}
\tau = 1 + a_{\tilde{\psi},\nu}.
\end{equation}
Therefore, (\ref{areavol}) and Lemma \ref{l:volume} imply the identity
\[
\widehat{\Cal}(\tilde{\psi}) = \frac{1}{\mathrm{area}(\Sigma,\omega_0)} \int_{\Sigma} (\tau - 1)\, \omega_0 = \frac{\vol(Y,\lambda)}{\vol(Y,\lambda_0)}-1,
\]
which proves (i). Moreover, if $z\in \mathrm{int}(\Sigma)$ is a contractible fixed point of $\tilde\psi$, then the Reeb orbit
\[
\gamma_z(t):= \phi_{\lambda}^t(j(z))
\]
is different from $\gamma_1$ and closes up at time $\tau(z) = 1 + a_{\tilde{\psi},\nu}(z)$. This number belongs to the interval $(1-\epsilon,1+\epsilon)$ when $\|R_{\lambda} - R_{\lambda_0}\|_{C^2}$ is small enough, again by Lemma \ref{l:nearly_Besse_normalized}(iv). Conversely, if $\epsilon$ is small enough then any closed orbit of $R_{\lambda}$ other than $\gamma_1$ and with (non necessarily minimal) period in the interval $(1-\epsilon,1+\epsilon)$ correesponds to an interior fixed point of $\phi= \pi(\tilde\psi)$. All fixed points of $\phi$ are contractible as fixed points of the lift $\tilde{\psi}$, as this is $C^1$-close to the identity. This proves (ii).

If $\tilde{\psi}$ is the identity, then every orbit of the Reeb flow of $\lambda$ is closed and, since the action $a_{\tilde{\psi},\nu}$ vanishes identically, has (non necessarily minimal) period 1 by (\ref{tauaction}). Therefore, $(Y,\lambda)$ is Besse with orbits having common period 1. Conversely, if $(Y,\lambda)$ has this property then the fact that $\tau$ is close to 1 and the closeness of $\tilde\psi$ to the identity imply that $\tilde{\psi}$ is the identity. This proves (iii).
\end{proof}

\begin{rem}
The above result can be generalized to a more general situation in which the Reeb flows of $\lambda$ and $\lambda_0$ have more closed orbits $\gamma_1, \dots,\gamma_h$ in common and the boundary of $\Sigma$ is mapped onto their union, but with a caveat: If $R_{\lambda}|_{\gamma_i} = c_i R_{\lambda_0}|_{\gamma_i}$ then the flux of the Hamiltonian isotopy defining 
$\tilde{\psi}$ is in general non zero, unless all numbers $c_i$ coincide. 
\end{rem}

\section{Proof of Theorem~\ref{t:main}}
\label{s:proof}

\begin{proof}[Proof of Theorem~\ref{t:main}$(i)$]
Let $\lambda$ be a contact form on $Y$ such that there exists a point $z\in Y$ whose Reeb orbit is open or has minimal period strictly larger than $\tau_k(\lambda)$. The same must be true for all points in a sufficiently small compact neighborhood $U\subset Y$ of $z$. Let $f:Y\to(-\infty,0]$ be a non-positive smooth function supported in $U$ and such that $f(z)<0$. For each $\epsilon>0$ small enough, the contact form $\lambda_\epsilon:=e^{\epsilon f}\lambda$  satisfies $\fix(\phi_{\lambda_\epsilon}^t)=\fix(\phi_{\lambda}^t)$ for all $t\in[0,\tau_k(\lambda)]$. In particular, $\tau_k(\lambda_\epsilon)=\tau_k(\lambda)$. Since 
\begin{align*}
 \vol(Y,\lambda_\epsilon) = \int_Y e^{2\epsilon f} \lambda\wedge\diff\lambda <
 \int_Y  \lambda\wedge\diff\lambda = \vol(Y,\lambda),
\end{align*}
we have that $\rho_k(\lambda_\epsilon)>\rho_k(\lambda)$, and therefore $\lambda$ is not a local maximizer of $\rho_k$. This proves that each local maximizer of $\rho_k$ is a Besse contact form $\lambda_0$ such that $k_0(\lambda_0)\leq k$.

Now, let $\lambda_0$ be a Besse contact form on $Y$ with $k_0:=k_0(\lambda_0)$. It remains to show that $\lambda_0$ is not a local maximizer of $\rho_k$ for any $k> k_0$. Without loss of generality, we can assume that $\tau_{k_0}(\lambda_0)=1$, so that the Reeb flow of $\lambda_0$ defines a locally free $S^1=\R/\Z$-action on $Y$ and
\[
\tau_k(\lambda_0) = \tau_{k_0}(\lambda_0) = 1 \qquad \forall k\geq k_0.
\]
 We denote by $\gamma_1,\dots,\gamma_h$ the singular orbits of $R_{\lambda_0}$ and by $\alpha_1,\dots,\alpha_h$ the integers greater than 1 such that $\gamma_i$ has minimal period $1/\alpha_i$ (if $\lambda_0$ is Zoll, we have $h=0$). Then
\[
k_0 = \alpha_1 + \dots + \alpha_h - h + 1.
\]
We denote by $B:=Y/S^1$ the quotient orbifold and by $\pi:Y\to B$ the quotient projection. We choose a small open disk $D\subset B$ with smooth boundary such that, for all $b\in D$, the preimage $\pi^{-1}(b)$ is a closed Reeb orbit of minimal period 1. We can identify $D$ with the open disk of radius $\rho$ in $\C$ and assume that the restriction of $\lambda_0$ to $\pi^{-1}(D)$ has the form
\begin{equation}
\label{e:form}
\lambda_0 = ds + \frac{r^2}{2}\, d\theta , \qquad \forall (re^{i\theta},s)\in D \times S^1,
\end{equation}
where $r,\theta$ are polar coordinates on $\C$. We now choose a smooth function $h:B \rightarrow \R$ such that
\begin{enumerate}[(i)]
\item $\int_Y h\circ \pi\, \lambda_0 \wedge d\lambda_0=0$;
\item $h$ equals a positive constant $c_+$ on $Y\setminus D$;
\item on $D$, $h$ has the form $h=\chi(r)$ where $\chi:[0,\rho]\rightarrow \R$ is a smooth function such that $-c_-:=\chi(0)<0$, $\chi(\rho)=c_+$ and $\chi'(r)>0$ for every $r\in (0,\rho)$.
\end{enumerate} 
For every $\epsilon>0$, we consider the 1-form
\[
\lambda_{\epsilon}:= (1+\epsilon \,h\circ \pi) \lambda_0,
\]
which is a contact form for $\epsilon$ small enough. By (i), we have
\begin{equation}
\label{e:volume_e}
\vol(Y,\lambda_{\epsilon}) = \int_Y (1+\epsilon \,h\circ \pi)^2 \, \lambda_0 \wedge d\lambda_0 = \vol(Y,\lambda_0) + c\, \epsilon^2,
\end{equation}
where
\[
c := \int_Y (h\circ \pi)^2\, \lambda_0 \wedge d\lambda_0.
\]
Let $\epsilon>0$ be so small that $\lambda_{\epsilon}$ is a contact form. Condition (ii) implies that the set $\pi^{-1}(B\setminus D)$ is invariant under the Reeb flow of $\lambda_{\epsilon}$, and hence the same is true for its complement $\pi^{-1}(D)$. The Reeb orbits of $\lambda_{\epsilon}$ on $\pi^{-1}(B\setminus D)$ are exactly the Reeb orbits of $\lambda_0$ reparametrized in such a way that their period gets multiplied by $1+\epsilon c_+$. In particular, on $\pi^{-1}(B\setminus D)$ the Reeb flow of $\lambda_{\epsilon}$ has exactly
\[
k_0 - 1 =  \alpha_1 + \dots + \alpha_h - h
\]
closed orbits with period strictly less that $1+ \epsilon c_+$:  the iterates $\gamma_i^m$ with $1\leq m \leq \alpha_i-1$.

On $\pi^{-1}(D)$, the Reeb flow of $\lambda_{\epsilon}$ has an orbit of minimal period $1-\epsilon c_-$, which is given by the inverse image by $\pi$ of the center of $D$, and all other orbits are either non-periodic or have a very large minimal period when $\epsilon$ is small. The latter assertion follows from (\ref{e:form}) and (iii), which imply that the Reeb orbits of $\lambda_{\epsilon}$ in $\pi^{-1}(D)$ are lifts of Hamiltonian orbits on $D$ defined by the standard symplectic form $r\, dr\wedge d\theta$ and the radial Hamiltonian $\chi$. These orbits
wind around the circle of radius $r\in (0,\rho)$ with frequency $\frac{\epsilon \chi'(r)}{2\pi r}>0$, which by the mean value theorem has the upper bound
\[
\frac{\epsilon \chi'(r)}{2\pi r} \leq \frac{\epsilon}{2\pi} \max_{r\in [0,\rho]} |\chi''(r)|, \qquad \forall r\in (0,\rho).
\]
If $\epsilon$ is so small that the above upper bound is smaller than $(1+\epsilon c_+)^{-1}$ and
\[
2(1-\epsilon c_-) \geq 1+\epsilon c_+,
\]
we conclude that on $\pi^{-1}(D)$ the Reeb flow of $\lambda_{\epsilon}$ has precisely one closed orbit whose period is strictly less that $1+ \epsilon c_+$. 

Summing up, $\lambda_{\epsilon}$ has $k_0$ many orbits whose period is strictly less that $1+ \epsilon c_+$. Together with the fact that this Reeb flow has infinitely many closed orbits of minimal period $1+\epsilon c_+$, we deduce that
\[
\tau_k(\lambda_{\epsilon}) = 1+\epsilon c_+ \qquad \forall k> k_0,
\]
when $\epsilon>0$ is small enough. By (\ref{e:volume_e}), we conclude that for every $k> k_0$ the $k$-th systolic ratio of $\lambda_{\epsilon}$ has the lower bound
\[
\rho_k(\lambda_{\epsilon}) = \frac{\tau_k(\lambda_{\epsilon})^2}{\vol(Y,\lambda_{\epsilon})} = \frac{(1+\epsilon c_+)^2}{\vol(Y,\lambda_0) + c\, \epsilon^2} \geq \frac{1+2 \epsilon c_+}{\vol(Y,\lambda_0) + c\, \epsilon^2},
\]
which is strictly larger than
\[
\frac{1}{\vol(Y,\lambda_0)} = \rho_k(\lambda_0)
\]
if $\epsilon$ is small enough. This shows that $\lambda_0$ is not a local maximizer of $\rho_k$ in the $C^{\infty}$-topology if $k> k_0$.
\end{proof}

\begin{proof}[Proof of Theorem~\ref{t:main}$(ii)$]
Let $(Y,\lambda_0)$ be a Besse contact 3-manifold. We recall that the positive integer
\[
k_0:=k_0(\lambda_0)
\] 
is the minimal $k$ so that the Reeb orbits of $(Y,\lambda_0)$ have minimal common period $\tau_k(\lambda_0)$. Without loss of generality, up to multiplying $\lambda_0$ with a positive constant, we can assume that
\begin{align*}
\tau_{k_0}(\lambda_0)=1. 
\end{align*}

We first carry out the proof under the assumption that $(Y,\lambda_0)$ is not Zoll, so that $k_0>1$. We denote by $\gamma_1,...,\gamma_h$ the singular Reeb orbits of $(Y,\lambda_0)$, that is, the closed Reeb orbits with minimal period strictly less than $1$. 
We denote by $\alpha_i>1$ the positive integer whose reciprocal $1/\alpha_i$ is the minimal period of $\gamma_i$, 
and by $\gamma_i^m$ the closed Reeb orbit $\gamma_i$ seen as a $m/\alpha_i$-periodic orbit. Therefore, 
\begin{align}
\label{e:k_nonZoll}
k_0 = \alpha_1+...+\alpha_h - h + 1.
\end{align}
It is well known that all the periodic orbits $\gamma_i^m$ with $1\leq m\leq \alpha_i-1$ are non-degenerate, i.e.
\begin{align*}
\ker(\diff\phi_{\lambda_0}^{m/\alpha_i}(\gamma_i(0))-I)=\mathrm{span}\{R_{\lambda_0}(\gamma_i(0))\},\qquad\forall m=1,...,\alpha_i-1.
\end{align*}
We refer the reader to \cite[Section~4.1]{Cristofaro-Gardiner:2020aa} for a proof of this fact. 
Standard results about perturbation of vector fields imply that, for every 
\[
\epsilon\in\big(0,\tfrac1{2\max\{\alpha_1,...,\alpha_h\}}\big),
\]  
there is a $C^3$-neighborhood $\VV$ of $\lambda_0$ such that every $\lambda\in\VV$ satisfies the following properties.

\begin{enumerate}
\item[(i)] $R_{\lambda}$ has pairwise distinct closed orbits $\tilde\gamma_i$, $i=1,\dots,h$, such that $\tilde{\gamma}_i$ has minimal period in $(\tfrac1{\alpha_i}-\epsilon,\tfrac1{\alpha_i}+\epsilon)$ and is $C^2$-close to $\gamma_i$.
\item[(ii)] The family of possibly iterated closed orbits of $R_{\lambda}$ of period less than or equal to $1-\epsilon$ is
\[
 \big\{ \tilde\gamma_i^m\ \big|\ i\in\{1,...,h\},\ m\in\{1,...,\alpha_i-1\} \big\}.
\]
\end{enumerate}
Here, we say that two closed curves $\gamma:\R/p\Z \rightarrow Y$ and $\tilde{\gamma}:\R/\tilde{p} \Z \rightarrow Y$ are $C^2$-close if $\gamma$ and $\tilde{\gamma}$ are $C^2$-close on $[0,\max\{p,\tilde{p}\}]$.

In particular, for any $\lambda\in\VV$, the closed Reeb orbit $\gamma_1$ of $(Y,\lambda_0)$ with minimal period $1/\alpha_1$ is $C^2$-close to some closed Reeb orbit $\tilde\gamma_1$ of $(Y,\lambda)$ with minimal period $T\in(\tfrac1{\alpha_1}-\epsilon,\tfrac1{\alpha_1}+\epsilon)$. Up to multiplying the contact form $\lambda$ with a constant close to $1$, we can assume that $T=1/\alpha_1$, that is, $\gamma_1$ and $\tilde\gamma_1$ have the same minimal period $1/\alpha_1$. By an argument analogous to \cite[Prop.~3.10]{Abbondandolo:2018fb}, there exists a diffeomorphism $\upsilon:Y\to Y$ such that $\upsilon\circ\gamma_1=\tilde\gamma_1$ and the quantities $\|\upsilon^*\lambda - \lambda_0\|_{C^2}$ and
$\|R_{\upsilon^*\lambda}-R_{\lambda_0}\|_{C^2}$ are small. Therefore, up to pulling back $\lambda$ by $\upsilon$, we can assume that
\begin{align*}
R_\lambda|_{\gamma_1}=R_{\lambda_0}|_{\gamma_1},
\end{align*}
that is, $\gamma_1$ is a closed orbit of minimal period $1/\alpha_1$ for both $R_{\lambda}$ and $R_{\lambda_0}$. After this modification, we can assume that $\lambda$ and $R_{\lambda}$ are arbitrarily $C^2$-close to $\lambda_0$ and $R_{\lambda_0}$ respectively, so that the assumptions of Proposition \ref{p:sum_up} are fulfilled.

The contact form $\lambda$ satisfies
\begin{align*}
 \tau_{k_0}(\lambda)\leq \tau_{k_0}(\lambda_0)=1;
\end{align*}
as a consequence of~\eqref{e:k_nonZoll}, of point~(ii) above and of the fact that $\gamma_1^{\alpha_1}$ is an orbit of $R_{\lambda}$ of period 1. If $\vol(Y,\lambda)>\vol(Y,\lambda_0)$, we have $\rho_{k_0}(\lambda)<\rho_{k_0}(\lambda_0)$, and we are done. Therefore, it remains to consider the case in which
\begin{align}
\label{e:vol_lambda_lambda_0}
 \vol(Y,\lambda)\leq\vol(Y,\lambda_0).
\end{align}

We now apply Proposition~\ref{p:sum_up} (using the objects and terminology introduced therein), choosing a $C^1$-neighborhood $\UU\subset\widetilde{\Ham}_0(\Sigma,\omega_0)$ of the identity such that the conclusion of the fixed point Theorem~\ref{t:fixed_point} with $c:=\frac{1}{2}$ holds for all elements of $\UU$. We require $\lambda$ and $R_{\lambda}$ to be sufficiently $C^2$-close to $\lambda_0$ and $R_{\lambda_0}$ respectively, so that the element $\tilde\psi\in\widetilde{\Ham}_0(\Sigma,\omega_0)$ provided by Proposition~\ref{p:sum_up} is contained in $\UU$. By Proposition~\ref{p:sum_up}(i) and~\eqref{e:vol_lambda_lambda_0}, the normalized Calabi invariant of $\tilde\psi$ has the value
\begin{align*}
\widehat{\Cal}(\tilde\psi)=\frac{\vol(Y,\lambda)}{\vol(Y,\lambda_0)} - 1 \leq0.
\end{align*}
By Theorem~\ref{t:fixed_point}, $\tilde\psi$ has a contractible fixed point $z\in \mathrm{int}(\Sigma)$ whose normalized action satisfies
\begin{align}
\label{e:action_upper_bound}
 a_{\tilde{\psi}}(z) + \frac{1}{2} a_{\tilde\psi}(z)^2 
 \leq \frac{1}{2} \, \widehat{\Cal}(\tilde\psi)
 =
 \frac12 
 \left(\frac{\vol(Y,\lambda)}{\vol(Y,\lambda_0)}-1\right).
\end{align}
Moreover, if $\tilde\psi$ is not the identity, then the above inequality is strict.

Assume first that $\tilde\psi$ is not the identity. By Proposition~\ref{p:sum_up}(ii), the contact manifold  $(Y,\lambda)$ has a closed Reeb orbit of period $1+a_{\tilde\psi}(z)\in(1-\epsilon,1+\epsilon)$. By the strict inequality in~\eqref{e:action_upper_bound}, we obtain the desired strict upper bound
\begin{align*}
\rho_{k_0}(\lambda)
&=
\frac{\tau_{k_0}(\lambda)^2}{\vol(Y,\lambda)}
\leq
\frac{(1+a_{\tilde\psi}(z))^2}{\vol(Y,\lambda)}
=
\frac{1+2a_{\tilde\psi}(z)+a_{\tilde\psi}(z)^2}{\vol(Y,\lambda)}\\
&<
\frac{1+\frac{\vol(Y,\lambda)}{\vol(Y,\lambda_0)}-1}{\vol(Y,\lambda)}
= \frac{1}{\vol(Y,\lambda_0)} = \frac{\tau_{k_0}(\lambda_0)^2}{\vol(Y,\lambda_0)}   = \rho_{k_0}(\lambda_0).
\end{align*}

Assume now that $\tilde\psi$ is the identity. By Proposition~\ref{p:sum_up}(iii), $(Y,\lambda)$ is Besse and its Reeb orbits have common period 1. Since the only closed Reeb orbits of $(Y,\lambda)$ with minimal period less than $1$ are $\tilde\gamma_1,...,\tilde\gamma_h$, we infer that $1$ is the minimal common period of the closed Reeb orbits of $(Y,\lambda)$. Every $\tilde\gamma_i$ is $C^2$-close to the corresponding $\gamma_i$ and has minimal period close to the minimal period $1/\alpha_i$ of $\gamma_i$. This, together with the Besse property, implies that $\tilde\gamma_i$ and $\gamma_i$ have the same minimal period (once again, provided $\lambda$ is sufficiently $C^3$-close to $\lambda_0$). We conclude that the $C^2$-close Besse flows of $\lambda$ and $\lambda_0$ have the same common period 1 and there is a period preserving bijection between their singular orbits. Thanks to the local rigidity of Seifert fibrations, we can find a diffeomorphism $\theta:Y\to Y$ such that $\theta^*R_{\lambda}= R_{\lambda_0}$. Deforming $\theta$ by means of a Moser's trick, we can actually ensure  that $\theta^* \lambda=\lambda_0$.

Actually, the existence of a diffeomorphism $\theta:Y\to Y$ with the latter property follows also from a theorem of Cristofaro-Gardiner and the third author, stating that the prime action spectrum determines Besse contact forms on closed 3-manifolds. Here, the prime action spectrum $\sigma_{\mathrm{p}}(\lambda)$ is the set of minimal periods of the Reeb orbits of $\lambda$, and the above discussion implies in particular that $\sigma_{\mathrm{p}}(\lambda) = \sigma_{\mathrm{p}}(\lambda_0)$. According \cite[Theorem~1.5]{Cristofaro-Gardiner:2020aa}, the equality $\sigma_{\mathrm{p}}(\lambda) = \sigma_{\mathrm{p}}(\lambda_0)$ implies the existence of a diffeomorphism $\theta:Y\to Y$ such that $\theta^*\lambda=\lambda_0$, also without assuming $\lambda$ to be close to $\lambda_0$.

It remains to consider the case in which $(Y,\lambda_0)$ is Zoll, for which $k_0=1$. This case  was already treated by Benedetti-Kang \cite{Benedetti:2021aa}, generalizing the result of the first author together with Bramham-Hryniewicz-Salom\~ao \cite{Abbondandolo:2018fb} for the special case $Y=S^3$, but we add some details here for the reader's convenience. The argument provided above in the non-Zoll case goes through in the Zoll case as well, except for the existence of the closed orbit $\tilde\gamma_1$, which now cannot be obtained perturbatively starting from a non-degenerate orbit of $R_{\lambda_0}$ as in (i) above. Since all orbits of $R_{\lambda_0}$ have the same minimal period 1, we choose $\gamma_1$ to be any one of them. This is the orbit we will apply Proposition~\ref{p:sum_up} to. Note that if $\gamma$ is any other orbit of $R_{\lambda_0}$, then we can find a diffeomorphism $\eta_\gamma: Y \rightarrow Y$ such that $\eta_{\gamma}^* \lambda_0 = \lambda_0$ and $\eta_{\gamma}\circ \gamma = \gamma_1$. Moreover, the set of these diffeomorphisms can be chosen to be pre-compact in the $C^k$-topology for every $k\in \N$.

We now consider a perturbation $\lambda$ of $\lambda_0$. If $\lambda-\lambda_0$ is $C^3$-small, then $R_{\lambda}$ admits a closed orbit $\tilde{\gamma}_1$ of period close to 1 which is $C^2$-close to some orbit $\gamma$ of $R_{\lambda_0}$. 
This is a consequence of the fact that the space of $1$-periodic closed Reeb orbits of the Zoll contact form $\lambda_0$ is Morse-Bott non-degenerate (see, e.g., \cite{Weinstein:1973}, \cite{Bottkol:1980} or \cite{Ginzburg:1987lq}). Up to replacing $\lambda$ by $\eta_{\gamma}^* \lambda$, which is still $C^3$-close to $\lambda_0=\eta_{\gamma}^* \lambda_0$, we may assume that $\tilde{\gamma}_1$ is $C^2$-close to $\gamma_1$. The rest of the proof continues as in the non Zoll case.
\end{proof}

\begin{rem}
\label{newrem}
There is a key point in which the proof of Theorem~\ref{t:main}(ii) above differs from the proofs of the local systolic maximality of Zoll contact forms in  \cite{Abbondandolo:2018fb} and \cite{Benedetti:2021aa}. The proofs from these two papers use the weaker version (\ref{e:weaker_fix_point}) of the fixed point Theorem \ref{t:fixed_point}, and in this case it is crucial that the boundary of the global surface of section is given by a closed orbit of $\lambda$ having minimal period. In the Besse case, the same argument would require us to have the boundary of the global surface of section on an orbit $\gamma$ which realizes $\tau_{k_0}(\lambda)$, where $k_0=k_0(\lambda_0)$. This orbit might be close to a singular orbit of $\lambda_0$, and hence be one of the orbits that are considered in assertion (i) of the above proof, but could also be an orbit of minimal period close to 1 bifurcating from the set of regular orbits of $\lambda_0$. In the latter case, finding a global surface of section with boundary on $\gamma$ and first return map $C^1$-close to the identity seems problematic:
we could apply a diffeomorphism bringing this orbit to a fixed regular orbit of $R_{\lambda_0}$, but we cannot hope to have a uniform bound on the $C^k$ norms of this diffeomorphism, because the set of regular orbits of $\lambda_0$ is not compact and $\gamma$ could be very close to some iterate of a singular orbit of $\lambda_0$. This issue is overcome by the more precise fixed point Theorem \ref{t:fixed_point} which we proved here, whose use does not require the boundary periodic orbit to have any minimality property.
\hfill\qed
\end{rem}

\end{document}